\documentclass{amsart}
\usepackage{appendix}
\usepackage{graphicx}
\usepackage[parfill]{parskip} 
\usepackage{amsfonts, amscd, mathrsfs, amsmath, amssymb,
amsthm, bm}
\usepackage{url}
\usepackage{hyperref} 
\hypersetup{backref,pdfpagemode=FullScreen,colorlinks=true}
\usepackage{tikz-cd}
\usepackage{}
\usepackage{color, verbatim}
\usepackage[hmargin=3cm,vmargin=3cm]{geometry}
\numberwithin{equation}{section}

\newtheorem{theorem}[equation]{Theorem}

\newtheorem{proposition}[equation]{Proposition}
 
\newtheorem{lemma}[equation]{Lemma}

\newtheorem{conjecture}{Conjecture}
\newtheorem{definition}[equation]{Definition}
\theoremstyle{definition}

\newtheorem{terminology}{Terminology}

\theoremstyle{remark}

\newtheorem{question}{Question}

\DeclareFontShape{OMX}{cmex}{m}{b}{<-> cmexb10}{}
\SetSymbolFont{largesymbols}{bold}{OMX}{cmex}{m}{b}

\DeclareMathOperator {\area} {area}

\DeclareMathOperator{\lie}{\mathrm{lie}}

\DeclareMathOperator{\carea}{carea}
\begin{document}
\title{Quantum Maslov classes}
\author{Yasha Savelyev}
\thanks {}
\email{yasha.savelyev@gmail.com}
\address{University of Colima, CUICBAS}
\begin{abstract} We give a construction of ``quantum Maslov
characteristic classes'', generalizing to higher dimensional
cycles the Hu-Lalonde-Seidel morphism.
We also state a conjecture
extending this to an $A _{\infty}$ functor from the exact
path category of the space of monotone Lagrangian branes to the Fukaya category. Quantum Maslov classes
are used here for
the study of Hofer geometry of Lagrangian equators in $S
^{2}$, giving a rigidity phenomenon for the Hofer metric
2-systole, which
stands in contrast to the flexibility phenomenon of the
closely related Hofer metric girth studied by Rauch ~\cite{cite_Itamar}, in
the same context of Lagrangian equators of $S ^{2}$.  
More applications appear in
~\cite{cite_SavelyevGlobalFukayacategoryII}.  
\end{abstract}
\maketitle
\section{Introduction}
There are numerous results in symplectic geometry which
express a kind of symplectic flexibility vs rigidity
phenomena. But one novelty here is that this
rigidity/flexibility will involve Hofer geometry of Lagrangian
submanifolds. In fact, we will see this just for the simplest
example of Lagrangian equators of $S ^{2}$. 

Our main method is a higher dimensional generalization of
the Hu-Lalonde-Seidel morphism
~\cite{cite_HuLalondeArelativeSeidelmorphismandtheAlbersmap},
that we call Quantum Maslov classes. 

Let $Lag (M)$ denote the space of 
monotone Lagrangian branes: in particular closed, oriented, spin, monotone
Lagrangian submanifolds of a symplectic manifold $(M, \omega
)$.  
Denote by $\mathcal{P} (L _{0}, L _{1})$ the space of
exact paths in $Lag (M)$ from $L _{0}$ to $L _{1}$, with its natural $C
^{\infty }$ topology, see Section \ref{sec_ExactPaths}.
We set $\Omega  _{L _{0}} Lag (M) := \mathcal{P} (L _{0},
L _{0})  $.  The 
Hu-Lalonde-Seidel morphism in one basic form is a 
homomorphism:
\begin{equation} \label{eq_HuLalondeIntro}
 \pi _{0}  (\Omega  _{L _{0}} Lag (M)) \to FH (L
 _{0}, L _{0}),
\end{equation}
where $FH (L _{0}, L _{0})$ denotes the $\mathbb{Z}  _{2}$ graded
Floer homology group.

\begin{theorem} \label{thm_}
For all $d$, there is a natural and non-trivial additive group homomorphism:
\begin{equation*}
\Psi _{d}: \pi _{d}   (\mathcal{P} (L _{0}, L _{1})) \to FH
(L _{0}, L _{1}).  
\end{equation*}
\end{theorem}

The extension furthermore,  determines a certain functor to the
Donaldson-Fukaya category. But we will not give details
of this here, as it is not needed for the main geometric
application. See
Conjecture \ref{conj_functor} for a formal
statement, on the level of an $A _{\infty}$ functor to the
Fukaya category.

We will use non-triviality of $\Psi$ for our Hofer geometric application.
This story is formally similar to author's use of quantum characteristic
classes in Hofer geometry, 
~\cite{cite_SavelyevQuantumcharacteristicclassesandtheHofermetric}.
However, unlike the case of
~\cite{cite_SavelyevQuantumcharacteristicclassesandtheHofermetric},
when we try to apply the calculation of $\Psi$ to the study of Hofer geometry
of Lagrangians, an unexpected new complexity
arises. To get anything interesting (for $d>0$) we need certain tautness conditions on families of
Lagrangians. This gives rise to new geometric structures
called taut Hamiltonian structures, whose theory we partly
develop here.

\subsection{Taut conditions and Hofer geometry} \label{sec_Taut conditions and Hofer geometry}
The basic example of a taut condition is the following.
\begin{definition}\label{def_concordantIntro}
Two smooth loops $$p _{0}, p _{1}: S ^{1} \to  Lag (M)  $$
are said to be \textbf{\emph{taut concordant}} if the following holds: 
\begin{itemize}
	\item There is a smooth fiber-wise Lagrangian sub-fibration $$\mathcal{L}
	\subset Cyl \times M, \quad Cyl=S ^{1} \times [0,1], $$
	satisfying $\mathcal{L} _{(\theta, 0)} = p _{0} (\theta)
	$ and $\mathcal{L} _{(\theta, 1)} = p _{1} (\theta)
	$ for all $\theta \in S ^{1}$.
	\item There is a Hamiltonian connection $\mathcal{A}$ on
	the trivial $M$ bundle $Cyl \times M \to Cyl$, preserving $\mathcal{L}$, such that
	the coupling form $\Omega _{\mathcal{A}} $ of $\mathcal{A}$ vanishes on $\mathcal{L}$. See Section \ref{sec:couplingform} for the definition of coupling forms.
\end{itemize}
\end{definition}

We let $\Omega   ^{taut} _{L _{0}} Lag (M) \subset
\Omega  _{L _{0}}  Lag (M) $ denote the subspace of  loops taut concordant to $p _{L _{0}}$, the constant loop at $L _{0}$.

\subsubsection{Lagrangian
equators in $S ^{2}$} \label{sec_Systolic inequality} Let $L
_{0} \subset S ^{2}$ now denote the standard equator. And
let $Lag ^{eq} (S ^{2}) \subset Lag (S ^{2}, L _{0})$ be the
subspace of great circles, where $Lag (M, L _{0}) \subset
Lag (M)$ denotes the subspace of elements Hamiltonian
isotopic to $L _{0}$. Note that $Lag ^{eq} (S ^{2})$ is
naturally diffeomorphic to $S ^{2}$. 

The
natural embedding $$i: \Omega  _{L _{0}}  Lag ^{eq} (S ^{2})
\hookrightarrow  \Omega   _{L _{0}} Lag (S ^{2}),$$ has
image in $\Omega   ^{taut} _{L _{0}} Lag ( S ^{2})$ as two loops $p
_{0}, p _{1}: S ^{1} \to Lag ^{eq} (S ^{2} ) $ are taut concordant iff they are homotopic, see Lemma
\ref{lemma:tautLagS2}. 

Let $$f: S ^{2} \to \Omega _{L _{0}}
^{taut} Lag (S ^{2}),$$ represent the image by $i _{*}$ of
the generator of $$\pi _{2} (\Omega _{L _{0}} ^{taut} Lag (S
^{2}), p _{L _{0}}) \simeq \pi _{3}  (S ^{2}) \simeq \mathbb{Z}.
$$
Let $b = [f]$ denote its homotopy class (based at $p _{L
_{0}}$).
\begin{theorem} \label{thm:Hofer}
    We have the following identity for the 2-systole with respect to $L ^{+} $:
   \begin{equation*}
  \min_{f' \in b} \max _{s \in S ^{2} } L ^{+} (f' (s)) = 1/2 \cdot \area (S ^{2},
  \omega),
  \end{equation*} where $L ^{+} $ denotes the positive Hofer length functional, as defined in Section \ref{section:hoferlength}. 
Furthermore, $$\pi _{3}  (Lag
^{eq} (S ^{2}))  \to \pi _{3}
(Lag (S ^{2})), 
\text{ is an injection}. $$
\end{theorem}
For contrast, suppose we measure a related quantity of the
``girth'' (infimum of the diameter of a representative) of
the generator $[g]$ of $\pi _{2} (Lag (S ^{2}), L _{0}),$ 
as in Rauch~\cite{cite_Itamar}. Then there is an upper bound for
this girth, which is smaller than the lower bound for girth
of $[g]$, considered as an element of $\pi _{2} (Lag ^{eq}
(S ^{2} ), L _{0})$.  
Indeed, it may be that girth of the generator $[g] \in \pi
_{2} (Lag (S ^{2}), L _{0})$ is actually 0. I think the
latter is unlikely, it would in particular contradict the
Lagrangian version of the injectivity radius conjecture of ~\cite{cite_LalondeSavelyevOntheinjectivityradiusinHofergeometry}. 

To summarize, passing
from classical equators to Lagrangian equators, we see
a squeezing phenomenon for girth.
On the other hand, our theorem says that this kind of
squeezing cannot happen at all for the 2-systole, provided
we work with taut families. In other words, whereas the
2-systole exhibits
Hamiltonian rigidity, the girth in \cite{cite_Itamar} while
geometrically, closely related, exhibits flexibility.  

In fact the construction of
Rauch should also show the flexibility of the 2-systole as soon as
we remove the tautness assumption.
That is:
\begin{equation*}
  \min_{f' \in  [f], [f] \in \pi _{2} (\Omega _{L _{0}} Lag (S
	^{2}))} \max _{s \in S ^{2} } L ^{+} (f' (s)) < 1/2 \cdot
	\area (S ^{2}, \omega).
\end{equation*}

\begin{question} \label{que_} Is there a non-zero lower
bound?
\end{question}
Another question:
\begin{question} \label{que_semiclassical} It is shown in
~\cite{cite_SavelyevBottperiodicityandstablequantumclasses}
that a certain semi-classical limit of quantum
characteristic classes are the Chern classes. Is there some
semi-classical limit for the quantum Maslov classes?
\end{question}

\section{Hamiltonian fibrations and taut structures}
We collect here some preliminaries on moduli spaces of
holomorphic sections of fibrations with Lagrangian boundary
constraints, and the closely related curvature bounds.
There is an apparently new theory here of taut Hamiltonian
structures,  but aside from that much of this material has
previously appeared elsewhere, perhaps in less generality.
Some of the generality here is a slightly excessive for the
purpose of this paper, but it is used in
~\cite{cite_SavelyevGlobalFukayacategoryII}, and it does not
substantially increase volume. 
\subsection {Coupling forms} \label{sec:couplingform}
We refer the reader to  \cite[Chapter
6]{cite_McDuffSalamonIntroductiontosymplectictopology} for more details on what
follows. We will suppose throughout that $(M, \omega )$ is a closed
symplectic manifold, and later on monotone. A Hamiltonian fibration is a smooth fiber bundle
$$M \hookrightarrow P \to X,$$ 
with structure group $\mathcal{H} = \operatorname  {Ham}(M, \omega)
$ with its $C ^{\infty} $ Frechet topology.
A \textbf{\emph{Hamiltonian connection}} is just an
Ehresmann $\mathcal{H} $ connection on this fiber bundle.
$\mathcal{H} $ trivializations will be called Hamiltonian
bundle trivializations. $\mathcal{H} $ bundle maps will be
called Hamiltonian fibration maps or Hamiltonian bundle
maps.

A \emph{coupling form}, as defined in
\cite{cite_GuilleminLermanEtAlSymplecticfibrationsandmultiplicitydiagrams},
for a Hamiltonian
fibration $M \hookrightarrow P \xrightarrow{p} X$, is a closed 2-form $
{\Omega} $ on $P$ whose restriction to fibers coincides with $\omega $  and
that has the property: 
\begin{equation*}  \int _{M}  {\Omega} ^{n+1} =0,
\end{equation*}
with integration being integration over the fiber operation.

Such a 2-form determines a Hamiltonian
connection, by declaring horizontal spaces to be $ {\Omega}
$-orthogonal spaces to the vertical tangent spaces. A coupling form generating a
given connection $\mathcal{A}$ is unique. A Hamiltonian connection $
\mathcal {A} $ in turn determines a coupling form $ {\Omega} _{
\mathcal {A}} $ as follows. First we ask that $ {\Omega} _{ \mathcal
{A}} $ induces the connection $ \mathcal {A} $ as above. This determines $
{\Omega} _{ \mathcal {A}} $ up to values on $ \mathcal {A}
$-horizontal lifts $ \widetilde{v},  \widetilde{w} \in T _{p} P $ of $v,w \in T
_{x} X $. We specify these values by the formula
\begin{equation} \label{eq:couplingvalue}
{\Omega} _{
\mathcal {A}} ( \widetilde{v}, \widetilde{w}) = R _{ \mathcal {A}} (v, w) (p),
\end{equation}
where $R _{\mathcal {A}}$ is the lie algebra valued curvature 2-form of $\mathcal{A}$. Specifically, for each $x$, $R _{\mathcal {A}}| _{x} $ is 
a 2-form valued in $C
^{\infty} _{norm} (p ^{-1} (x))$ - the space of 0-mean normalized smooth
functions on $p ^{-1} (x) $.

\subsection {Hamiltonian structures on fibrations} \label{section:exact}
Let $S$ be a Riemann surface with boundary, with punctures
in the boundary, and a fixed structure of strip end charts
at ends, (positive or negative), i.e. a strip end structure
as in ~\cite[Section 3.2]{cite_SavelyevGlobalFukayacategoryI}. 

Let $M
\hookrightarrow \widetilde{{S}} \xrightarrow{pr} {S} $ be
a Hamiltonian fibration, with model fiber a monotone
symplectic manifold $(M,\omega)$, with distinguished
Hamiltonian bundle trivializations $$[0,1] \times (0, \infty) \times M \to  \widetilde{{S}} $$ at the positive ends,
and with distinguished Hamiltonian bundle trivializations $$[0,1] \times (-\infty, 0) \times M \to \widetilde{{S}}, $$ at the negative ends.
These are collectively called called \textbf{\emph{strip end
charts}}, (slightly abusing terminology).  Given the
structure of such Hamiltonian bundle trivializations we say that $\widetilde{{S}} $ has
   \textbf{\emph{end structure}}.
\begin{definition} \label{def:respectsendstructure}
    Let  $$\mathcal {L} \subset (\widetilde{S}| _{\partial
		S}= pr ^{-1} (\partial S))   \to \partial S$$ be
		a Lagrangian sub-bundle, with model fiber a Lagrangian
		brain, or more specifically an object as in ~\cite[Section 5]{cite_SavelyevGlobalFukayacategoryI}, 
		(in particular a closed, spin, oriented, monotone
		Lagrangian submanifold). 
		We say that $\mathcal{L}$ \textbf{\emph{respects the end
   structure}} if $\mathcal{L}$ is a constant sub-bundle in
	 the strip end chart trivializations above.
\end{definition}
For $\mathcal{L} $ as above, in the strip end chart coordinates at the end $e
_{i} $, let $L ^{j} _{i}  $ denote the fibers (which are by assumption $t$ independent) of
$\mathcal{L}$ over 
$$\{j\} \times \{t\}, \,j=0,1.  $$

We say that a Hamiltonian connection $\mathcal {A}$ on $\widetilde{{S}} $ is \textbf{\emph{compatible}}
with the connections $\{\mathcal{A} _{i} \}$ on $[0,1]
\times M$ at each end $e _{i} $, if the following holds. 
In the strip coordinate chart at the $e _{i} $ end,
$\mathcal{A}$ is flat and $\mathbb{R}$-translation invariant
and has the form $\overline{\mathcal{A}} _{i}$,
where $\overline {\mathcal{A}} _{i}   $ denotes the $\mathbb{R}$-translation invariant extension of $\mathcal{A} _{i} $ to $(0, \pm \infty) \times \mathbb{R}
\times M$, depending  on whether the end is positive or
negative. We say that a Hamiltonian connection, $\mathcal
{A}$ on $\widetilde{{S}} $  is $\mathcal{L}$-\textbf{\emph{exact}} if $\mathcal{A}$ preserves
$\mathcal{L}$ (this means that the horizontal spaces of $\mathcal{A}$ are tangent to $\mathcal{L}$). 

For $\mathcal{A}$ compatible with $\{\mathcal{A} _{i} \}$ as above, a family $\{j _{z} \}$ of fiber wise $\omega$-compatible almost complex structures on $\widetilde{S} $ will be said to \textbf{\emph{respect the end structure}} if the following holds. 
At each end $e _{i} $, in the strip end chart above, the family $\{j _{z}\} $ is $\mathbb{R}$-translation invariant
and is admissible with respect to $\mathcal{A} _{i}$, in the
sense of ~[Definition 5.3]\cite{cite_SavelyevGlobalFukayacategoryI}.
The data $\Theta =(\widetilde{S}, S, \mathcal{L}, \mathcal{A}, \{j _{z} \} )$, with $\mathcal{A}$
compatible with $\{\mathcal{A} _{i} \}$, $\{j _{z} \}$,
and respecting the end structure,  will be called a \textbf{\emph{Hamiltonian structure}}.

We will normally suppress $\{j _{z} \}$ in the notation and elsewhere for simplicity, as it will be purely in
the background in what follows, (we do not need to
manipulate this family explicitly).

\subsection{Relative section classes of Hamiltonian structures} \label{sec_Relative section classes of Hamiltonian structures}
\begin{definition} Let $(\widetilde{S}, S, \mathcal{L},
\mathcal{A})$ be a Hamiltonian structure, we say that
a smooth section $\sigma$ of $\widetilde{S} \to S $ is
\textbf{\emph{asymptotically flat}} if the following holds. 
At each end $e _{i} $ of $S$, $\sigma$ $C ^{1} $-converges to an $\mathcal{A}$-flat section. Specifically, in the strip end chart at a positive end, this means that there is a $\mathcal{A}$-flat section $$\widetilde{\sigma}: [0,1] \times (0,\infty) \to [0,1] \times (0,\infty) \times M,$$ so that for every $\epsilon>0$ there is a $t>0$ so that $$d _{C ^{1} } (\widetilde{\sigma}, {\sigma}| _{[0,1] \times [t,\infty)} ) < \epsilon. $$ (Similarly for a negative end.)
\end{definition} 
Note that the above definition implies that $$\lim _{s
\mapsto \infty}  {\sigma| _{[0,1] \times
\{s\} }} = \gamma ^{i}, $$ for some $\mathcal{A}
_{i}$-flat sections $\gamma _{i}$ of $[0,1] \times M $,  where the limit
is the $C ^{1}$ limit. (Similarly for negative ends.) 
So we can say that $\sigma$ 
is \textbf{\emph{asymptotic}} to $\gamma ^{i}$ at the $e _{i}$
end, and that $\gamma ^{i}$ is the
\textbf{\emph{asymptotic constraint}}  of $\sigma$ at the $e
_{i}$ end.
\begin{definition}\label{def:relativeclass}
	Given a pair of asymptotically
flat sections $\sigma _{1}, \sigma _{2}  $, with boundary in
$\mathcal{L}$, we say that they have the same
\textbf{\emph{relative class}} if:
\begin{itemize}
	\item They are asymptotic to the same flat sections at
	each end. (In the sense above.) 
	\item They are homologous relative to the boundary
	conditions and relative to the asymptotic constraints at the ends.
\end{itemize}
  The set of relative classes will be denoted by $H _{2}
	^{sec} (\widetilde{S}, \mathcal{L}) $. 
\end{definition}
\subsection{Families of Hamiltonian structures.} \label{section:Family}
\begin{definition} \label{def:HamiltonianStructure}
 A \textbf{\emph{family Hamiltonian structure}}, consists of the following:
\begin{enumerate}
      \item A  smooth, connected, simply connected, compact,
			oriented manifold $\mathcal{K}$, possibly with
			boundary or corners.    
      \item For
each $r \in \mathcal{K}$ a Hamiltonian structure
$(\widetilde{S} _{r}, S _{r}, \mathcal{L} _{r}, \mathcal{A}_{r}  )$, 
 such that this family fits into smooth fibrations $$ \widetilde{S}
 \hookrightarrow \bm{\widetilde{S}} \xrightarrow{p _{1}
 } \mathcal{K},  \quad S \hookrightarrow \textbf{S} \xrightarrow{p} \mathcal{K},  $$
and $\{\widetilde{S} _{r} \}$, respectively $\{S _{r} \}$
correspond to the fibers of the first respectively second
fibration. In addition the following:
\begin{itemize}
	\item The second fibration has fiber a Riemann surface, so that $\{S _{r} \}= \{p  ^{-1} r  \}$.
	\item The first fibration has fibers $p
	_{1} ^{-1} (r)$  that are themselves the total spaces of smooth Hamiltonian fibrations $M \hookrightarrow \widetilde{S}
	_{r} \to S _{r} $, ($\widetilde{S} _{r} \simeq
	\widetilde{S} $),
	such that the structure group of $\widetilde{S}
	\hookrightarrow  \bm{\widetilde{S}}
	\xrightarrow{p _{1} } \mathcal{K}$ can be reduced to
	smooth Hamiltonian fibration maps (of $\widetilde{S}$).
\end{itemize}
     To elaborate further, let $$M \hookrightarrow
		 \widetilde{S} \to S$$ be a Hamiltonian fibration
		 over a Riemann surface $S$. Let $Aut$ denote the group
		 of Hamiltonian fibration automorphisms of
		 $\widetilde{S} $. Then $\bm{\widetilde{S}}
		 \xrightarrow{p _{1} } \mathcal{K}$ is the associated
		 bundle $P \times _{Aut} \widetilde{S}$ for some principal $Aut$ bundle $P$ over $\mathcal{K}$.
      \item The strip end charts $$e _{i,r}:  [0,1] \times (0, \infty) \times M \to 
\widetilde{{{S}  }} _{r},  $$ for the positive ends, fit
into a Hamiltonian fibration map onto the image:
\begin{equation} \label{eqE}
\widetilde{e}  _{i}:  [0,1]
   \times (0, \infty) \times \mathcal{K} \times M \to
	 \bm{\widetilde{S}},
\end{equation}
similarly for the negative ends.
\item \label{item:rinvariant} In case of positive ends, we then have an induced smooth $r$-family of connections $\{e _{i,r} ^{*}  \mathcal{A} _{r} \}$
on $[0,1] \times (0, \infty) \times M$, and an induced smooth $r$-family of Lagrangian subfibrations $\{e _{i,r} ^{-1} \mathcal{L}_{r} \}   $ over $\partial [0,1]
   \times (0,  \infty)$.
We ask that $$\forall r: \{e _{i,r} ^{-1}
   \mathcal{L}_{r} \} = \{0\} \times (0,  \infty) \times L ^{0} _{i} \cup \{1\} \times (0,  \infty) \times L ^{1} _{i},        $$
where $L ^{j} _{i}  $ are as following the Definition
\ref{def:respectsendstructure}.
Furthermore, we ask that $$\forall r: \{e _{i,r} ^{*}
\mathcal{A} _{r} \} = \overline {\mathcal{A}}  _{i} $$ for
$\mathcal{A} _{i}, \overline{\mathcal{A}}_{i}    $ as
previously.  (Similarly for negative ends.) 
   \item \label{property:A} There is a Hamiltonian
	 connection $\mathcal{A}$ on $\bm{\widetilde{S}} \to
	 \textbf{S} $ that extends all the connections
	 $\mathcal{A} _{r} $ (in the natural sense), and preserves $\textbf{L} := \cup _{r} \mathcal{L} _{r}  $. 
\end{enumerate}  
\end{definition}
We will write  
$\{\widetilde{{S}}_{r}, {S}_{r}, \mathcal{L}_{r},
\mathcal{A} _{r} \} _{\mathcal{K}} $  for this data, $\mathcal{K}$ may be omitted from notation when it is implicit. 
\begin{terminology} 
We will usually just Hamiltonian structure instead of
family Hamiltonian structure. The distinction between the
two is clear from context and notation.
\end{terminology}
   
Let $\{\widetilde{{S}}_{r}, {S}_{r}, \mathcal{L}_{r},
\mathcal{A} _{r} \} $ be a Hamiltonian
structure. In the notation above, if in addition there exists
a Hamiltonian connection $\mathcal{A}$ on
$\bm{\widetilde{S}} \to \textbf{S} $ as in Property \ref{property:A}, so that $\Omega _{\mathcal{A}}$ vanishes on $\textbf{L}$, we will say that $\{\widetilde{{S}}_{r}, {S}_{r}, \mathcal{L}_{r}, \mathcal{A} _{r} \}  $ is a \textbf{\emph{hyper taut Hamiltonian structure}}.

%
\subsection {Moduli spaces of sections of Hamiltonian
structures} \label{sec:ModuliSpacesHamStructures}
Let $\Theta = (\widetilde{{S}}, S,  \mathcal{L},
\mathcal{A} )$ be a Hamiltonian structure.
For a section $\sigma$ of $\widetilde{S} $ define its
vertical $L ^{2} $ energy or \textbf{\emph{Floer energy}}  by $$e (\sigma):= \int _{S} |\pi _{vert}  \circ d\sigma| ^{2}, $$ $$\pi _{vert}: T \widetilde{S} \to T ^{vert} \widetilde{S}    $$ is the $\mathcal{A}$-projection, for $T ^{vert} \widetilde{S}$ the vertical tangent bundle of $\widetilde{S} $, that is the kernel of the projection $T \widetilde{S} \to TS$. 

As in ~\cite{cite_SavelyevGlobalFukayacategoryI}, let $J (\mathcal{A} ) $ denote the almost
complex structure on $\widetilde{S} $ determined by
$\mathcal{A} $ and $\{j _{z}\}$ as follows.
\begin{itemize}
	\item $J (\mathcal{A} ) $ preserves the $\mathcal{A}
	$-horizontal distribution of $\widetilde{S} $.
	\item The projection map $\widetilde{S} \to S$ is $J (A)
	$-holomorphic.
	\item  The restriction of $J (\mathcal{A} ) $ to each
	fiber $M _{z}$ of $\widetilde{S} $ over $z \in S$ is $j
	_{z}$.
\end{itemize}
We say that $J (\mathcal{A} ) $ is \textbf{\emph{induced}}
by $\mathcal{A} $, $\{j _{z}\}$.

Define  $\overline{\mathcal{M}} (\Theta)$ to be the
Gromov-Floer compactification of the space of  $J
(\mathcal{A})$-holomorphic sections $\sigma$ of
$\widetilde{\mathcal{S}} $, with finite Floer energy, and with
boundary on $\mathcal{L}$. 

Note that for any
$J({\mathcal{A}})$-holomorphic $\sigma$ we have an identity:
\begin{equation*}
e (\sigma) = \int _{S} \sigma ^{*} \Omega _{\mathcal{A}}.
\end{equation*}
We leave to the reader to verify that $\Omega
_{\mathcal{A}}$ vanishes on $\mathcal{L}$, using 
the conditions that $\mathcal{A}$ preserves
$\mathcal{L}$, and that $\mathcal{L} $ is
a fiber-wise Lagrangian distribution. So the standard energy
controls apply, to deduce the standard Gromov-Floer
compactification structure on spaces of $J (\mathcal{A}
)$-holomorphic sections, at least once we establish Lemma
\ref{lemmaInvariantPairing} further ahead.

\subsubsection{Family version} \label{sec_Family version}
More generally, if $\{\Theta _{r} \}=\{\widetilde{{S}}_{r}, S _{r},  \mathcal{L}_{r}, \mathcal{A}_{r}\} _{\mathcal{K}}$ is a Hamiltonian structure, let $$\overline{\mathcal{M}} (\{\Theta _{r}\})$$ be the Gromov-Floer compactification of the space of pairs $(\sigma, r)$, $r \in \mathcal{K} $ with $\sigma$  a $J (\mathcal{A} _{r})$-holomorphic, finite Floer energy section of $\widetilde{\mathcal{S}}_{r} $, with
boundary on $\mathcal{L}
_{r}$.

Pick an Ehresmann fiber bundle connection $\mathcal{B} $ on $\widetilde{S}
 \hookrightarrow \bm{\widetilde{S}} \xrightarrow{p _{1}
 } \mathcal{K}$, so that the $\mathcal{B} $-holonomy maps at $r _{0} \in
 \mathcal{K} $ are morphisms of the Hamiltonian structure
 $\widetilde{S}  _{r _{0}}$ in the natural sense
 (Hamiltonian fibration maps of $\widetilde{S} $ preserving the
 Lagrangian distribution $\mathcal{L} $, and such that the
 bundle map is trivial over the ends, with respect to the
 strip end structure).
Since
$\mathcal{K} $ is simply connected the action of the holonomy
group of $\mathcal{H} $ on $H _{2} ^{sec} (\widetilde{S} _{r
_{0}}, \mathcal{L} _{r _{0}}) $ is
trivial. It follows that the groups $H _{2} ^{sec}
(\widetilde{S} _{r }, \mathcal{L} _{r})$ are naturally identified for
various $r$, using the connection $\mathcal{B} $. And we just write $A \in H _{2} ^{sec}
(\widetilde{S}, \mathcal{L}) $ for an element.

Then we denote by $$\overline{\mathcal{M}} ( \{\Theta _{r}
\},A) \subset \overline{\mathcal{M}} (\{\Theta _{r}  \})$$
the subset corresponding to relative class $A \in H _{2}
^{sec} (\widetilde{S}, \mathcal{L}) $ curves.

Let $\{\Theta _{r} =  (\widetilde{{S}}_{r}, S _{r},  \mathcal{L}_{r},
\mathcal{A}_{r})\}$ be a Hamiltonian structure, then for each end $e _{i} $ of $S _{r}$ we have a Floer chain complex  $$
CF (\mathcal{A} _{i}):= CF (L ^{0} _{i}, L ^{1} _{i}, \mathcal{A} _{i}, \{j _{z} \} ),$$ 
(independent of $r$ by part \ref{item:rinvariant} of
Definition \ref{def:HamiltonianStructure}) generated over
$\mathbb{Q}$  by $\mathcal{A} _{i} $-flat sections of $[0,1]
\times M$, with boundary on $L ^{0} _{0}, L ^{0} _{1}  $.
This chain complex is defined as in
~\cite[Section 6.1]{cite_SavelyevGlobalFukayacategoryI}.  
\begin{definition}
We say that $\{\Theta _{r} \}$ is
$A$-\textbf{\emph{regular}} if:
\begin{itemize}
	\item The pairs  $(\mathcal{A} _{i}, \{j _{z} \} )$ are
	regular so that the Floer chain complexes $CF (\mathcal{A}
	_{i})$ are defined.
	\item ${\mathcal{M}} (\{\Theta _{r} \},  A)$  is regular,
(transversely cut out).
\end{itemize}
And we say that $\{\Theta _{r} \}$ is \textbf{\emph{regular}} if it is $A$-regular for all $A$. We say that $\{\Theta _{r} \}$ is $A$-\textbf{\emph{admissible}} if 
there are no elements $$(\sigma, r) \in \overline{\mathcal{M}} (\{\Theta _{r} \},  A),$$  for $r$ in a neighborhood of the boundary of $\mathcal{K}$. 

\end{definition}

\begin{definition}  \label{definition:concordance} 
Given a pair $\{\Theta _{r} ^{i}  \} = \{\widetilde{{S}} _{r} ^{0} , {S}_{r} ^{0},
      \mathcal{L}_{r} ^{0}, \mathcal{A} _{r} ^{0}\} _{\mathcal{K}} $, $i=1,2$, of Hamiltonian structures
we say that they are \textbf{\emph{concordant}} if the
following holds. 
There is a Hamiltonian structure $$\{\mathcal{T} _{r} \}= \{\widetilde{{T}}  _{r},
{T}_{r} , \mathcal{L}'_{r}, \mathcal{A}' _{r}
   \} _{\mathcal{K} \times [0,1]} ,$$ with
an oriented diffeomorphism (in the natural sense, preserving all structure)
$$  \{\widetilde{{S}} _{r} ^{0} , {S}_{r} ^{0} ,
      \mathcal{L}_{r} ^{0}, \mathcal{A} _{r} ^{0} 
      \} _{ \mathcal{K} ^{op}}     \sqcup \{\widetilde{{S}} _{r} ^{1}, {S}_{r} ^{1} ,
      \mathcal{L}_{r} ^{1}, \mathcal{A} _{r} ^{1}
      \} _{ \mathcal{K}} 
      \to  \{\widetilde{{T}} _{r},
{T} _{r}, \mathcal{L}'_{r}, \mathcal{A}' _{r
      } \} _{\mathcal{K} \times \partial I} ,$$ 
      where $op$ denotes the opposite orientation for $\mathcal{K}$.
\end{definition}
\begin{definition} We say that a Hamiltonian structure 
   $\{\Theta _{r} \} = \{\widetilde{{S}} _{r}, {S}_{r},
   \mathcal{L}_{r}, \mathcal{A} _{r} \} _{\mathcal{K}} $ is
	 \textbf{\emph{taut}} if the following holds. For any pair
	 $r _{1}, r _{2} \in \mathcal{K}  $, $\Theta _{r _{1} }$
	 is concordant to $\Theta _{r _{2}} $ by a concordance
	 $\{\widetilde{{T}}  _{\tau}, {T}_{\tau},
	 \mathcal{L}'_{\tau}, \mathcal{A}' _{\tau} \} _{[0,1]}$, which is a hyper taut Hamiltonian structure.
\end{definition}

\begin{definition}  \label{definition:isotopy} 
Given an $A$-admissible pair $\{\Theta _{r} ^{i}  \} $, $i=1,2$, of Hamiltonian structures, we say that they are $A$-\emph {\textbf{admissibly concordant}} if there
is an $A$-admissible Hamiltonian structure $$\{\widetilde{{T}}  _{r},
{T}_{r} , \mathcal{L}'_{r}, \mathcal{A}' _{r}
   \} _{\mathcal{K} \times [0,1]}, $$ which furnishes a concordance. If this concordance is in addition a taut Hamiltonian structure, then we say that these pairs are \textbf{\emph{$A$-admissibly taut concordant}}. 
\end{definition}
\begin{lemma} \label{lemma:gluing} 
Let $\{\Theta _{r}\} = \{\widetilde{{S}} _{r}, {S}_{r}, \mathcal{L}_{r}, \mathcal{A} _{r} \}$ be $A$-regular and $A$-admissible, with $S _{r} $ having one distinguished
negative end $e _{0} $, and let $\gamma _{0} $ be the
asymptotic constraint of $A$ at the $e _{0} $ end. 
Define $$ev _{A} = ev (\{\Theta _{r} \}, A)  = \# {\mathcal{M}} (\{\Theta _{r} \},
A) \cdot \gamma _{0} \in CF (\mathcal{A} _{0}), $$ 
where $\# {\mathcal{M}} (\{\Theta _{r} \},  A)$ means signed count of elements when the dimension is 0, and is otherwise set to be zero.
Furthermore, suppose that $CF (\mathcal{A} _{0} ) $ is perfect.
Then the count $\# {\mathcal{M}} (\{\Theta _{r} \},
A)$ depends only on the $A$-admissible concordance class of
$\{\Theta _{r} \}$, that is the homology class of $ev _{A}$ is
an invariant of the $A$-admissible concordance class of
$\{\Theta _{r}\}$.
\end{lemma}
\begin{proof}
Suppose we are given an $A$-admissible concordance (which we may assume to be regular)
$$\mathcal{T} = \{\widetilde{{T}}  _{r}, {T}_{r},
\mathcal{L}'_{r}, \mathcal{A}' _{r} \} _{\mathcal{K} \times
[0,1]},  $$ 
   between Hamiltonian structures $\{\Theta _{r} ^{0}\}$ and $\{\Theta _{r} ^{1}\}$.  
Then we get a one dimensional compact moduli space
${{\mathcal{M}}} (\mathcal{T}, A)$. By assumption
on the perfection of $CF (\mathcal{A} _{0})$, boundary contributions from Floer degenerations cancel out, so that the 
boundary is:
\begin{equation*}
\partial {{\mathcal{M}}} (\mathcal{T} , A) = {\mathcal{M}}
(\{\Theta ^{0} _{r} \} ^{op}, A)  \sqcup  {\mathcal{M}}
(\{\Theta ^{1} _{r} \}, A)
\end{equation*}
where $op$ denotes opposite orientation. From which the result follows.
\end{proof}

\begin{definition}\label{def_EvaluationTotal}
Let $\{\Theta _{r}\} _{r \in
\mathcal{K} }$, be a Hamiltonian structure, with $\mathcal{K} $ possibly with boundary, with $S _{r} $ having one
distinguished negative end $e _{0} $. 
Define:
\begin{equation}\label{eq_}
  ev(\{\Theta _{r}\}) = \sum _{A} ev (\{\Theta _{r} \}, A)
	\in CF (\mathcal{A}  _{0}),
\end{equation}
where the sum is over all classes $A \in H _{2} ^{sec}
(\widetilde{S}, \mathcal{L} )$. 
\end{definition}

In general $ev(\{\Theta _{r}\})$ is not closed, but we have:
\begin{lemma} \label{lem_evaluation}
Let $\{\Theta _{r}\} _{r \in \mathcal{K} }$ be regular, with
$\mathcal{K} $ having no boundary, with $S _{r} $ having one
distinguished negative end $e _{0} $.
Define
\begin{equation}\label{eq_}
  ev(\{\Theta _{r}\}) = \sum _{A} ev (\{\Theta _{r} \}, A)
	\in CF (\mathcal{A}  _{0}),
\end{equation}
Then $ev(\{\Theta _{r}\})$ is a cycle, whose homology class
depends only on the concordance class of $\{\Theta _{r}\}$.
\end{lemma}
\begin{proof} [Proof]
We only sketch the proof since this is standard Floer
theory. The proof is formally identical to the standard
proof that Floer continuation maps are chain homotopy maps.

Set:
$$\mathcal{M}  ^{1} = \cup _{A} {\mathcal{M}} (\{\Theta _{r}\},
A),$$ where the sum is over all $A$ such that the expected
dimension of ${\mathcal{M}} (\{\Theta _{r}\},
A)$ is one. As usual this is a finite sum by monotonicity.
Then $\mathcal{M} ^{1}$ is a one dimensional oriented
manifold with boundary.

Then $$\#\partial \mathcal{M}  ^{1} =0,$$ where this is the signed
count of elements.
As $\mathcal{K} $ has no boundary, this is a signed
count of holomorphic buildings (broken flow lines) with
a pair of components. One component is an
element $\sigma $ of ${\mathcal{M}} (\{\Theta _{r}\}, A')$,
for some $A'$, s.t. the expected dimension of ${\mathcal{M}}
(\{\Theta _{r}\}, A')$ is zero, i.e.  contributing to
$ev(\{\Theta _{r}\})$.  The other component
corresponds to a contribution to the Floer boundary of $\gamma
_{A '}$, where the latter is the asymptotic constraint of
$A'$. 

As usual all such boundary contributions to $\partial \mathcal{M}  ^{1}$ that can happen do happen (by a standard gluing argument). It then readily
follows, since $\#\partial \mathcal{M}  ^{1} = 0$, that the Floer
differential of $ev(\{\Theta _{r}\})$ is zero.

The second part of the lemma is analogous. Let 
$$\mathcal{T} = \{\widetilde{{T}}  _{r}, {T}_{r},
\mathcal{L}'_{r}, \mathcal{A}' _{r} \} _{\mathcal{K} \times
[0,1]},  $$ be a concordance (which we may assume to be regular) between Hamiltonian structures $\{\Theta _{r}
^{0}\}$ and $\{\Theta _{r} ^{1}\}$, with the latter as in
the hypothesis of the lemma.

Set $$\mathcal{M}  ^{1} = \cup _{A} {\mathcal{M}} (\mathcal{T}, A),$$
Then again this is a one dimensional oriented manifold with
boundary. Analyzing  $\partial \mathcal{M}  ^{1}$, and since $\#M ^{1} =0$, we get that $$
\partial \, {ev}   ((\{\mathcal{T} \}))  = ev(\{\Theta ^{1} _{r}\}) - ev(\{\Theta ^{0}
_{r}\}).$$ And this finishes the proof.



\end{proof}

\subsection {Area of fibrations} \label{section:areafib}
\begin{definition} \label{def:area}
For a Hamiltonian connection $\mathcal{A} $ on a bundle $M
\hookrightarrow \widetilde{S}  \to S$, with $S$ a Riemann
surface, define a 2-form $\alpha _{\mathcal{A} }$ on $S$ by:
\begin{equation} \label{eq:alphaA} \alpha _{\mathcal{A}}  (v, jv) :=  |R _{\mathcal{A}} (v,jv)| _{+},
\end{equation}
where $v \in T _{z} S $, $R _{\mathcal{A}} (v,w)$ as before
identified with a zero mean smooth function on the fiber
$\widetilde{S} _{z}$ over $z$ and where $|\cdot| _{+} $ is
operator: $|H| _{+} = \max _{\widetilde{S} _{z}} H,$ i.e.
the ``positive Hofer norm''.
\end{definition} 
And define
\begin{align} \label{eqDefArea}
\area  (\mathcal{A})  :=  \int _{S} \alpha _{\mathcal{A}}.
\end{align} 
Note that if $\Omega _{\mathcal{A}} $ is the coupling form
of $\mathcal{A}$, as before, then
${\Omega} _{\mathcal{A}}   + \pi ^{*} (\alpha _{\mathcal{A}})$ is \emph{nearly
symplectic}, meaning that
\begin{equation*} 
   \forall z \in S \, \forall v  \in T _{z} S: ({\Omega} _{\mathcal{A}}   + \pi ^{*} (\alpha)) ( \widetilde{v},
\widetilde{jv}) \geq 0,
\end{equation*}
where $ \widetilde{v}, \widetilde{jv} $ are the
$\mathcal{A}$-horizontal lifts of  $v, jv \in T _{z} S $.

Note that $\area (\mathcal{A})$ could be infinite if there are no constraints on $\mathcal{A}$ at the ends.

\begin{lemma} \label{lemma:energypositivity1}
Let $(\widetilde{S}, S, \mathcal{L}, \mathcal{A} )$ be
a Hamiltonian structure. For $\sigma \in \overline{\mathcal{M}} (\widetilde{{S}}, S,  \mathcal{L},
\mathcal{A})$
 we have  
  \begin{equation*}
    - \int _{S} \sigma ^{*}
    {\Omega} _{\mathcal{A}}     \leq \area  
  (\mathcal{A}).
  \end{equation*}
 \end{lemma}
\begin{proof}
We have $$\int _{S} \sigma ^{*}
    ({\Omega} _{\mathcal{A}} + \pi ^{*} \alpha) \geq 0, $$ 
whenever ${\Omega} _{\mathcal{A}}   + \pi ^{*} (\alpha)$ is nearly symplectic, by the defining properties of $J _{\mathcal{A}} $ and by $\sigma$ being $J _{\mathcal{A}} $-holomorphic.
From which our conclusion follows.
\end{proof}

\begin{lemma} \label{lemmaInvariantPairing} 
Let $\{(\widetilde{S} _{t}, S _{t},
\mathcal{L}_{t}, 
\mathcal{A} _{t}) \} _{[0,1]} $ be a taut concordance.
Let $\sigma _{j} $, $j=0,1$ be asymptotically flat sections of $\widetilde{S} _{j}$ in relative class $A$. 
Then
%
$$
   - \int _{S _{1} } \sigma _{1}  ^{*}  {\Omega} _{\mathcal{A} _{1} }  =  - \int _{S _{0} } \sigma _{0}  ^{*}  {\Omega} _{\mathcal{A} _{0}},
$$
whenever both integrals are finite. In particular, for a Hamiltonian structure  $(\widetilde{S}, S,
\mathcal{L}, 
   \mathcal{A})$,
$\int _{S} \sigma  ^{*}  {\Omega} _{\mathcal{A}}$ depends only on the relative class of $A$, whenever the integral is finite.
%
%
%
%
 \end{lemma}
\begin{proof}
 By the hypothesis, there is a connection $\mathcal{A}$ on
 $\bm{\widetilde{S}} $, extending each $\mathcal{A} _{t}
 $ and such that $\Omega _{\mathcal{A}} $ vanishes on
 $ \textbf{L} \subset \bm{\widetilde{S}}. $  The first part then follows by Stokes theorem. 
Here are the details. For $\sigma _{j} $ as above and for
each end $e _{i} $, cut off the part of the section $\sigma
_{j} $ lying over $[0,1] \times (t _{\delta _{1}, \delta
_{2} }, \infty)$ in the corresponding strip end chart at the end. Here $t _{\delta _{1}, \delta _{2} } $ is such that $\sigma _{0}| _{[0,1] \times \{t\}}   $ is $C ^{1} $ $\delta _{1} $-close to $\sigma _{1}| _{[0,1] \times \{t\}} $ for all $t > t _{\delta _{1}, \delta _{2} } $ and for each end, and is such that $$\int _{[0,1] \times (t _{\delta _{1}, \delta _{2} }, \infty)} \sigma _{j} ^{*}| _{[0,1] \times (t _{\delta _{1}, \delta _{2} }, \infty)}  {\Omega} _{{\mathcal{A} _{j} }} < \delta _{2}, \, j=1,2,    $$ for each end $e _{i} $.
Call the sections with the ends cut off as above by $\sigma _{j} ^{\delta _{1}, \delta _{2}  }  $, they are sections over the compact surfaces $S ^{cut} _{j} $, with ends correspondingly cut off. Then by Stokes theorem, using that ${\Omega} _{{\mathcal{A}}}  $ is closed and using the vanishing of ${\Omega} _{{\mathcal{A}}} $ on $ \textbf{L}   $:   for each $\epsilon$ there exists $\delta _{1}, \delta _{2}  $ such that
   $$\int _{S ^{cut} _{1}} (\sigma _{1} ^{\delta _{1}, \delta _{2}}) ^{*} {\Omega} _{{ \mathcal{A}}} -   \int _{S ^{cut} _{0}} (\sigma _{0} ^{\delta _{1}, \delta _{2}}) ^{*} {\Omega} _{{ \mathcal{A}}} < \epsilon,$$
and $$\int _{S ^{cut} _{j}} (\sigma _{j} ^{\delta _{1}, \delta _{2}}) ^{*} {\Omega} _{{ \mathcal{A} _{j} }} - \int _{S _{j} } \sigma _{j}  ^{*}  {\Omega} _{\mathcal{A} _{j} }  <\epsilon, \, j=1,2.  $$

The last part of the lemma follows from the first. For if
$\mathcal{A}$ preserves $\mathcal{L}$ then $\Omega
_{\mathcal{A}} $ vanishes on $\mathcal{L}$, as previously
observed, and consequently
the corresponding constant concordance: $$\{(\widetilde{S}
_{t}, S _{t}, \mathcal{L}_{t}, \mathcal{A} _{t}) \}
_{[0,1]}, \quad \forall t \in [0,1]: (\widetilde{S}_{t}, S _{t}, \mathcal{L}_{t}, \mathcal{A} _{t}) = (\widetilde{S}, S, \mathcal{L}, \mathcal{A})   $$ is taut.
\end{proof}
\begin{definition}
For $\sigma$ a relative class $A$ section of $\Theta=(\widetilde{{S}}, {S}, \mathcal{L}, \mathcal{A}) $ let us call:
\begin{equation*}
-\int _{S} \sigma  ^{*}  {\Omega} _{\mathcal{A}},
\end{equation*}
the $\mathcal{A}$-\textbf{\emph{coupling area}} of $\sigma$,
denoted by $\carea (\Theta, \sigma)$,
 we may also write $\carea (\Theta, A)$ for the same quantity. By the lemma above this is an invariant of the taut concordance class of $\Theta$.
\end {definition}
\begin{definition} \label{def:small}
Given a Hamiltonian structure $\Theta=(\widetilde{S},S, \mathcal{L}, \mathcal{A})$  we will say that 
 $\Theta$ is  $A$-\textbf{\emph{small}} if $$\area (\Theta)<
\carea(\Theta,A).$$ 
\end{definition}

\begin{lemma} \label{lemma:smallempty}
   Suppose that $\Theta=(\widetilde{{S}}, {S}, \mathcal{L},
	 \mathcal{A})  $ is $A$-small then $\overline{\mathcal{M}}
	 (\Theta, A)$ is empty. Or for a contrapositive, if $\overline{\mathcal{M}} (\Theta, A)$ is non-empty then:
\begin{equation*}
\carea (\Theta,A) \leq \area (\Theta).
\end{equation*}
\end{lemma}
\begin{proof} 
This is just a reformulation of Lemma \ref{lemma:energypositivity1}.
\end{proof}
  
\subsection {Gluing Hamiltonian structures} \label{section:gluing}

Let $\mathcal{D}$ denote the Riemann surface which is topologically  $D ^{2} - z _{0}$, $z _{0} \in \partial D ^{2} $,
endowed with a choice of a strip end chart at the end (positive or negative depending on context). The complex structure $j$ here is as induced from $\mathbb{C}$ under the assumed embedding $D ^{2} \subset \mathbb{C} $.

Let $(\widetilde{S}, S, \mathcal{L}, \mathcal{A})$ be
a Hamiltonian structure. We may cap off some of the open
ends $\{e _{i} \} _{i=0} ^{n} $ of ${S}$, by gluing at the ends copies of $\mathcal{D} $ with oppositely signed end.
More explicitly, in the strip coordinate charts at some, say positive, end 
$e _{i} $ of ${S}$, 
excise $[0,1] \times (t, \infty)$ for some $t>0$, call the
resulting surface ${S} - e _{i} $. Likewise, excise the negative end of $\mathcal{D} $, call this surface $\mathcal{D} - end$. Then glue ${S} -e _{i}  $ with  $\mathcal{D}-end$, 
along their new smooth boundary components. 
Let us denote the capped off surface by ${S} ^{/i}  $.  

Since $\widetilde{{S}}$ is naturally trivialized at the
ends, we may similarly cap off $\widetilde{\mathcal{S}} _{r}
$ over the $e _{i} $ end by gluing with the bundle
$\mathcal{D} \times M $ at the end, obtaining a Hamiltonian $M$ bundle $\widetilde{{S}}  ^{/i}
$ over ${S} ^{/i}$. 

Moreover, we have a certain gluing operation of Hamiltonian structures. In the case of ``capping off'' as above we glue $\Theta = (\widetilde{S}, S, \mathcal{L}, \mathcal{A})$ with the Hamiltonian structure $\Theta'  = (\mathcal{D} \times M, \mathcal{D}, \mathcal{L}', \mathcal{A}')  $ at the $e _{i} $ end, provided $\mathcal{A}'$ is compatible with the connection $\mathcal{A} _{i} $, in the sense of Section \ref{section:exact}, and provided 
$\mathcal{L}$ is compatible with $\mathcal{L}'$. The latter
means that $L ^{j} _{i} = {L'} ^{j} _{i}   $ where these are
the Lagrangians corresponding to the strip end chart trivialization of $\mathcal{L}, \mathcal{L}'$ at the corresponding ends, as in Definition \ref{def:respectsendstructure}.

Let us name the result of this capping off $\Theta \# _{i} \Theta' $. The following is immediate:
\begin{lemma} \label{lemma:immediategluing} Suppose that $\{\Theta _{r} \} _{\mathcal{K}} $, $\{\Theta' _{r} \} _{\mathcal{K}} $ with $\Theta' _{r} = (\mathcal{D} \times M, \mathcal{D}, \mathcal{L}' _{r} , \mathcal{A}' _{r} ) $ are taut Hamiltonian structures. Then:
\begin{equation*}
\{\Theta _{r} \# _{i}   \Theta' _{r}  \} _{\mathcal{K}} 
\end{equation*}
is taut, whenever the gluing operation is well-defined, that is whenever we have compatibility of connections and Lagrangian sub-fibrations at the corresponding end.
\end{lemma}

\begin{definition} \label{defInducedLoop}
Let $\pi: \mathbb{R} \to [0,1]$ denote the continuous
retraction map, sending $(-\infty, 0]$ to $0$, and sending
$[1,\infty)$ to $1$. Assuming the end $e _{0}$  of
$\mathcal{D} $ is positive, and using the coordinates of the
strip end chart $e _{0}: [0,1] \times (0, \infty) \to
\mathcal{D} $, fix the following parametrization $\zeta$ of the
boundary of $\mathcal{D} $.
$\zeta: \mathbb{R} \to \partial \mathcal{D}$, satisfies $\zeta (t) \in \{0\} \times (0, \infty)$ for $t \in (-\infty, 0) $, and 
$\zeta (t) \in \{1\} \times (0, \infty)$ for $t \in 
 (1, \infty) $. Given a smooth exact path 
$$p: [0,1] \to
Lag (M) $$ constant near $0,1$, let $\mathcal{L} _{p} \subset \partial \mathcal{D} \times M  $ denote the Lagrangian subfibration over $\partial \mathcal{D}$, with fiber over $r \in \partial \mathcal{D}$ given by $p \circ \pi (r)$. We say that a Lagrangian subfibration $\mathcal{L}$ as above is
\textbf{\emph{determined by $p$}} if 
$\mathcal{L} = \mathcal{L} _{p}$, after a fixed choice of parametrization of boundary of $\mathcal{D}$ by $\mathbb{R}$.
(In the case the end of $\mathcal{D} $ is negative, the
above is meant to be analogous.) 
\end{definition}
A Hamiltonian connection $\mathcal{A}$ on $[0,1] \times M$
uniquely corresponds to a choice of a smooth function $H:
[0,1] \times M \to \mathbb{R}$, normalized to have mean zero
at each moment. For the holonomy path of $\mathcal{A}$ over
$[0,1] $ is a path $\phi _{\mathcal{A}}: [0,1] \to Ham
(M,\omega)  $, generated by a Hamiltonian $H:
[0,1] \times M \to \mathbb{R}$, and this uniquely determines
the connection. Conversely, $H$ uniquely  determines
a Hamiltonian connection with holonomy path generated by
$H$. We can say that $H$ \textbf{\emph{generates
$\mathcal{A}$}}. 
\begin{lemma} \label{lemma:areadisk} Let $p$ and
$\mathcal{L} _{p}  \subset \partial \mathcal{D} \times
M  $ be as in definition above with $L ^{\pm}
(\widetilde{p})=\rho$, where $\widetilde{p} $ is some lift
of $p$ to $\operatorname {Ham} (M,\omega)$, that is $p (t)
= \widetilde{p} (t) (p (0)) $. Let $\mathcal{A} _{0}   $ be
a Hamiltonian connection on $[0,1] \times M$, generated by
a Hamiltonian $H: [0,1] \times M \to \mathbb{R} ^{} $ with
$L ^{\pm} $ length $\kappa$, constant for $t$ near $0,1$.
Then there is a Hamiltonian connection
$ \widetilde{{\mathcal{A}}} ^{p}_{0} $ on $\mathcal{D}
\times M$,  preserving $\mathcal{L} _{p} $, compatible with
respect to $\mathcal{A} _{0} $, and satisfying $$area
(\widetilde {\mathcal{A}} ^{p} _{0} ) \leq \kappa + \rho. $$
The construction is natural in the sense that
$(\widetilde{p},\mathcal{A} _{0}) \mapsto
\widetilde{{\mathcal{A}}} ^{p}_{0}  $ can be made into
a smooth map (of Frechet manifolds).
\end{lemma}
\begin{proof} Let $q: [0,1] \to \operatorname {Ham}
(M,\omega)$ be the holonomy path of $\mathcal{A} _{0} $, $q
(0)=id$, generated by $H$. Let $\widetilde{p}  \cdot q$ be
the usual path concatenation in diagrammatic order, and $H'$ be its generating Hamiltonian.

   Define a coupling form ${\Omega}'$ on $D ^{2} \times M$:
\begin{equation*}
{\Omega}' =  \omega - d (\eta (rad) \cdot H' d \theta),
\end{equation*}
for $(rad, \theta)$ the modified angular coordinates on $D ^{2}  $, 
$\theta \in [0,1]$, $0 \leq rad \leq 1$, and 
$\eta: [0,1] \to [0,1]$ is a smooth function satisfying $$0 \leq \eta' (rad),$$
and 
\begin{equation} 
\eta (rad) = \begin{cases} 1 & \text{if } 1 -\delta \leq rad \leq 1 ,\\
rad ^{2}  & \text{if } rad \leq 1-2\delta,
\end{cases}
\end{equation}
for a small $\delta >0$.  
By an elementary calculation  $$\area  (\mathcal{A}') = L ^{+} (p \cdot q) = L ^{+} (p)  + L ^{+} (q) , $$ where $\mathcal{A}'$ is the connection induced by $\Omega'$.
 Set $$arc=\{(1,\theta) \in D ^{2} \,|\, 0 \leq \theta \leq
 1/2\}.$$ Let $arc ^{c} $ denote the complement of $arc$ in
 $\partial D ^{2} $.   Fix a smooth embedding $i: D ^{2}
 \hookrightarrow \mathcal{D}$ such that the following is
 satisfied (see Figure \ref{figure:arc}):
\begin{itemize}
	\item The image of the embedding contains
$\partial \mathcal{D} - end$, where $end$ is the image of the
   distinguished (say positive) strip end chart $$[0,1] \times (0,\infty) \to \mathcal{D}.$$
	 \item $i(arc) \subset end ^{c} $,
	 \item $i (arc ^{c}) \subset end$. 
\end{itemize}
 \begin{figure} [h]
 \includegraphics[width=1.5in]{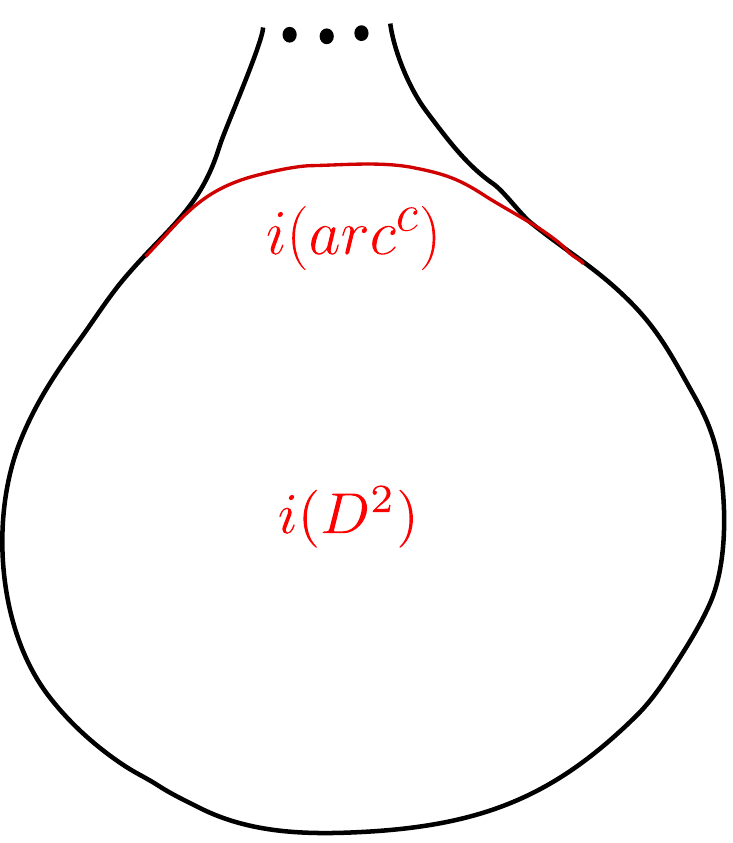}
 \caption {} \label{figure:arc}
\end{figure} 

Next fix a deformation retraction $ret$ of $\mathcal{D}$
onto  $i (D ^{2}) $, so that in the strip end chart above,
for $r \geq 1$ $ret$ is the composition $i \circ param \circ
pr$, where $$pr: [0,1] \times (0,\infty) \to [0,1]$$ the
projection and where $$param: [0,1] \to arc ^{c} \subset
D ^{2} $$ is a diffeomorphism.
   Finally, set ${\Omega} = ret ^{*} {\Omega}'$ on $ \mathcal{D} \times S ^{2} $, and set $\widetilde{\mathcal{A}} _{0} ^{p}  $ to be the induced Hamiltonian connection. As constructed $ \widetilde{{\mathcal{A}}} _{0} $ will be compatible with $\mathcal{A} _{0} $,  when the end of $\mathcal{D}$ is positive.  When the end is negative we take the reverse paths $p ^{-1}, q ^{-1}  $.
\end{proof}
Let us denote by $\Theta (p, \mathcal{A} _{0})
= (\mathcal{D} \times M, \mathcal{D},
\mathcal{L} _{p}, \widetilde{{\mathcal{A}}} ^{p}_{0})   $ the Hamiltonian structure as in the lemma above. When $p$ is the
constant map to $L$  we will instead write 
\begin{equation} \label{eq:ThetaL}
\Theta (L, \mathcal{A} _{0}). 
\end{equation}

The following says that under suitable conditions the
connection of the lemma above can be made to have $area$ 0.
\begin{lemma} \label{lemma:areadisk2} Let $H: M \times [0,1]
\to \mathbb{R} ^{} $ be a smooth time-dependent function
with zero mean at each moment. Let $p: [0,1] \to
\operatorname {Ham} (M,  \omega) $ be the path generated by
$H$. Let $L \in Lag (M, L _{0} )$,
and $p _{L}: [0,1] \to Lag (M, L _{0} )$ be the path $p _{L}
(t) = p (t) (L) $. Let $\mathcal{L} _{p} \in \partial
\mathcal{D} \times M$ be as in the Lemma
\ref{lemma:areadisk}. Suppose that $\mathcal{D} $ has the positive end $e _{0}$. And let $\mathcal{A} _{0}$ at the $e _{0}$ be generated by $H$, then there is a Hamiltonian connection
$ \widetilde{{\mathcal{A}}} ^{H} $ on $\mathcal{D}
\times M$,  preserving $\mathcal{L} _{p} $, compatible with
respect to $\mathcal{A} _{0} $, and satisfying $$area
(\widetilde {\mathcal{A}} ^{H} ) = 0.$$
The construction is natural in the sense that
$H \mapsto \widetilde{{\mathcal{A}}} ^{H}  $ can be made into a smooth map (of Frechet manifolds).
\end{lemma}
\begin{proof} We only sketch the proof as the idea is similar to the proof of Lemma \ref{lemma:areadisk}.
In the notation of the proof of Lemma \ref{lemma:areadisk}
let $r: i (D ^{2}) \to i (arc) $ be a smooth retraction.
Set $\mathcal{A}  := r ^{*} \mathcal{A} _{0}  $, and set
$\widetilde{\mathcal{A}} ^{H} = ret ^{*} \mathcal{A} $ (for $ret$ as
before).
\end{proof}
For future use, we denote by 
\begin{equation} \label{eq:ThetaH}
\Theta (H) = (\mathcal{D} \times M, \mathcal{D},
\mathcal{L} _{p}, \widetilde{{\mathcal{A}}} ^{H}),
\end{equation}
the Hamiltonian structure as in the lemma above. 

\subsubsection{Gluing Hamiltonian structures with estimates} \label{sec_Gluing Hamiltonian structures with estimates}
Now let $\Theta = (\widetilde{S}, S, \mathcal{L},
\mathcal{A})$ be a Hamiltonian structure.  For simplicity,
suppose that $\mathcal{L} $ is trivial with fiber $L _{0}$,
and that $\mathcal{A} $ is trivial over the boundary.
Suppose further that at the end $e _{i} $ the corresponding connection $\mathcal{A} _{i} $ is generated
by $L ^{\pm} $-length $\kappa _{i} $ Hamiltonian $H
_{i}$. By capping each $e _{i}$ end with $\Theta (L _{0},
\mathcal{A} _{i})  $ (keeping in mind that negative-positive
distinction)  we obtain a Hamiltonian structure we
call $\Theta ^{/} = (\widetilde{S} ^{/}, S, \mathcal{L}
^{/}, \mathcal{A} ^{/})  $.
By the lemma above:
\begin{equation} \label{eq:areakappai}
\area (\mathcal{A} ^{/}) \leq \area (\mathcal{A})
+ \sum _{i} \kappa   _{i}.
\end{equation}
 
Recall that a Lagrangian $L \subset M$ is called  \emph{monotone}  
with monotonicity constant $const>0$, if for any 
relative class $B \in H _{2} (M, L )$: $\omega (B) = const
\cdot \mu (B)$, where $\mu$ is the Maslov number. In what
follows $const$ is this monotonicity constant.
\begin{lemma} \label{lemma:gluinglowerbound} 
Let
$$\Theta:=\{\Theta _{r} \}:=\{\widetilde{S} _{r}, S _{r},
\mathcal{L} _{r} , \mathcal{A} _{r} \} _{\mathcal{K}} $$ be
a hyper taut Hamiltonian structure satisfying:
\begin{itemize}
	\item $\mathcal{L} _{r} $ is the trivial bundle with fiber
	$L _{0} $ for each $r$.
	\item $\mathcal{A} _{r} $ is the trivial connection over the
	boundary of ${S} _{r} $ for each $r$.
	\item The Floer chain complex  $CF(\mathcal{A} _{i} )$ is perfect for each $i$ and $\mathcal{A} _{i} $ is generated by a time dependent Hamiltonian $H _{i} $ with $L ^{\pm} $ length $\kappa _{i} $.
\end{itemize}
Let $\Theta ^{/} _{r}   = (\widetilde{S} ^{/} _{r}, S ^{/} _{r} , \mathcal{L} ^{/} _{r}, \mathcal{A} ^{/} _{r} ) $ be obtained from $\Theta _{r} $ by capping off each end 
$e _{i} $, so that \eqref{eq:areakappai} is satisfied.
For a given $A \in H _{2} ^{sec} (\widetilde{S}, \mathcal{L})  $, if 
\begin{equation*}
\forall r: \area (\mathcal{A} _{r} ) < \carea (\Theta _{r} ^{/}, A ^{/} )  - \sum _{i} \kappa _{i},
\end{equation*} 
where $A ^{/} $ is the capping off of $A$ is described in the proof, then 
$\overline{\mathcal{M}} (\{\Theta _{r}\} , A)$ is empty. 
Moreover, $$\forall r: \carea (\Theta _{r} ^{/}, A ^{/}) = -const \cdot Maslov ^{vert}   (A ^{/}),$$
where $Maslov ^{vert} $ is as in Appendix \ref{appendix:maslov}.
\end{lemma}
\begin{proof} 
Suppose otherwise that we have an element $(\sigma _{0},r _{0} ) \in \overline{\mathcal{M}} (\{\Theta _{r}\} , A)$. 
Suppose for the moment that $\overline{\mathcal{M}} (\{\Theta _{r}\} , A)$ is regular. 
There is a morphism (cf. Albers~\cite{cite_AlbersPSS})
$$PSS: QH (L) \to FH (L, L),$$ where the right-hand side is defined using our construction in terms of flat sections, 
and the left-hand side is interpreted for example as the
$\mathbb{Z} _{2}$ graded homology of the Pearl complex,
Biran-Cornea~\cite{cite_BiranCorneaLagQuantHomology}. Moreover, as shown by Albers this is an isomorphism in the present monotone context.

We won't give the full construction of this morphism in our
setting, as it just a reformulation of the construction
in \cite{cite_AlbersPSS}. Here is a quick sketch. 
Let $$\Theta _{-}  = (\mathcal{D} \times M, \mathcal{D}, \mathcal{L}, \mathcal{A} _{-} ), $$ be the Hamiltonian structure with $e _{0} $ being a negative end, $\mathcal{L}$ trivial with fiber $L$ (which is an object as before), and $$\mathcal{A} _{-}:= \widetilde{\mathcal{A}} ^{p=const}  _{0},  $$ 
with right-hand side as in Lemma \ref{lemma:areadisk}, for
$p$ being the constant path at $L$. Suppose that $\Theta _{-}$  is regular. 
Define $PSS ([L]) $ as the homology class of the Floer chain:
\begin{equation} \label{eq:CL}
 \sum _{A} ev (\Theta _{-}, A).
\end{equation}
where the sum is over all classes $A \in H _{2} ^{sec}
(\mathcal{D} \times M, \mathcal{L}). $ 

Now, for a general class $a \in QH (L) $, $PSS (a) $ is
defined similarly, but using the moduli space $\mathcal{M}
(\Theta _{-}, a, A) $. The latter can be defined as the subset  of $\mathcal{M} (\Theta _{-}, A) $ consisting of sections
intersecting a fixed smooth pseudocycle, see Zinger
~\cite{cite_ZingerPseudocycles}, representative of $a$.
More specifically, for $z _{0} \in \partial \mathcal{D}
$ let $\widetilde{S} _{z _{0}}$ be the fiber. Fix
a pseudo-cycle $g: B \to L \subset \widetilde{S}  _{z
_{0}}$ representing $a \in H _{2} (L) $. Then $\mathcal{M}
(\Theta _{-}, a, A) $ consists of elements of $\mathcal{M}
(\Theta _{-}, A) $ intersecting image of $g$. (Although we
use the language of pseudocycles in this outline, for
analysis it is technically simpler to use Morse homology and
Perl complex language as in ~\cite{cite_AlbersPSS}.) 

Now, the PSS morphism is an isomorphism in our monotone context,
and $CF(\mathcal{A} _{i} )$ is perfect for each $i$, by
assumption.  It follows
that the asymptotic constraint $\gamma _{i} $ of $\sigma
_{0} $ at each (positive) end $e _{i} $ satisfies:
\begin{equation*}
\langle \gamma _{i}, PSS (a)  \rangle =1,
\end{equation*}
for some $a$ uniquely determined. 
Moreover, Fredholm index and monotonicity restrictions
ensure that only a single class $A^i$ can contribute in
the sum. 
Let then $\sigma _{A ^{i}} \in \mathcal{M} (\Theta _{-}, a, A _{i}) $ be some element.

With this understanding, at each end \footnote {More
precisely at each positive end, but in the case of negative
end we do the analogous construction.} $e _{i} $, glue $\sigma
_{0}$ with $\sigma _{A ^{i}}$. We then obtain a $J (\mathcal{A} ^{/} _{r _{0} } )$-holomorphic, class $A ^{/} $ section $\sigma ^{/} _{0}$ of $\Theta ^{/} _{r _{0}}  $. 

By Lemma \ref{lemma:smallempty}:
\begin{equation*}
   \carea (\Theta _{r _{0} } ^{/}, A ^{/} ) \leq \area (\mathcal{A} _{r_0} ^{/}) \leq \area(\mathcal{A} _{r _{0} })  + \sum _{i} \kappa  _{i},  
\end{equation*}
so 
\begin{equation*}
   \carea (\Theta _{r _{0} } ^{/}, A ^{/} ) - \sum _{i} \kappa  _{i}  \leq \area (\mathcal{A} _{r _{0} }),
\end{equation*}
so that we contradict the hypothesis. So in the case $\overline{\mathcal{M}} (\{\Theta _{r}\} , A)$ is regular we are done with the first part of the lemma. When it is not regular instead of gluing just pre-glue to get a holomorphic building $\sigma ^{/} _{0}$, and the conclusion follows by the same argument.

To prove the last part of the lemma, note that each $\Theta
^{/} _{r} $ is taut concordant to $$\Theta _{0}= (D ^{2}
\times M, D ^{2}, \mathcal{L}, \mathcal{A} ^{tr} ),  $$ with
$\mathcal{L}$ trivial with fiber $L _{0} $,  and for
$\mathcal{A} ^{tr} $ the trivial connection. And $$\carea
(\Theta _{0}, \cdot) = {-}{const}\cdot Maslov ^{vert}
(\cdot),$$  as functionals on $H _{2} ^{sec}  (D ^{2} \times M, \mathcal{L}) $. It follows by Lemma \ref{lemmaInvariantPairing} that
$$\carea (\Theta ^{/}, \sigma ^{/}_{0} )= \carea (\Theta _{0}, \sigma ^{/} _{0} ) = -{const}\cdot Maslov ^{vert}  (\sigma ^{/} _{0}  ) =  -{const}\cdot Maslov (A ^{/}). $$

\end{proof}

\section {Quantum Maslov classes} 
\label{section:QuantumMaslov}
Quantum Maslov classes 
are relative analogues of quantum characteristic
classes in
\cite{cite_SavelyevQuantumcharacteristicclassesandtheHofermetric}.
The latter give Chern classes in a certain semi-classical limit 
~\cite{cite_SavelyevBottperiodicityandstablequantumclasses}.
So quantum Chern classes is perhaps the most suggestive name
for the construction in
\cite{cite_SavelyevQuantumcharacteristicclassesandtheHofermetric}.
The name ``quantum Maslov classes'' is then also meant to be suggestive,  as the classical Maslov numbers are
relative analogues of Chern numbers.  
   
One version of the relative Seidel morphism appears in Seidel's
\cite{cite_SeidelFukayacategoriesandPicard-Lefschetztheory}
in the exact case. This was 
further developed by 
Hu-Lalonde~\cite{cite_HuLalondeArelativeSeidelmorphismandtheAlbersmap}
in the monotone case.  

Let $Lag (M)$ be as in the Introduction. We may also denote
the component of $L$ by $Lag (M, L)$.  Then from our
perspective the Hu-Lalonde morphism is  a functor 
$$S: \Pi Lag (M)  \to DF (M),$$
where $\Pi Lag (M)$ is the category with objects
elements of $Lag (M)$, with $$hom _{\Pi Lag (M)} (L _{0},
L _{1}) = \pi _{0} (\mathcal{P} (L _{0}, L _{1})), $$  where
$\mathcal{P} (L _{0}, L _{1})$ is as in the Introduction.
Here, $DF (M)$ is the Donaldson-Fukaya category of $M$, see
also \cite{cite_CorneaBiranLagrangianCobordism},
\cite{cite_BiranCorneaLagCobGAFA}  which can be understood as an extension.
%

In our setup this works as follows. To an element of
$\mathcal{P} (L _{0}, L _{1})$ we have have an associated Lagrangian subbundle $\mathcal{L} _{p} $ of $\mathcal{D} \times M$ over the boundary, as in Definition \ref{defInducedLoop}.
Extend this to a Hamiltonian structure $$\Theta _{p}
= (\mathcal{D} \times M, \mathcal{D}, \mathcal{L} _{p},
\widetilde{\mathcal{A}} _{0} ^{p}  )$$ where $\widetilde{\mathcal{A}} _{0} ^{p}$ is as in Lemma \ref{lemma:areadisk}.

Assuming $\Theta _{p} $ is regular, we define $$S ([p]) \in
hom _{DF} (L _{0}, L _{1}) = FH (L _{0}, L _{1})$$ by
\begin{equation*}
   S ([p]) = [ev (\Theta _{p})],
\end{equation*}
where $ev (\Theta)$ is as in Definition
\ref{def_EvaluationTotal}. Since $\Theta _{p}$ is well
defined up to concordance, $S ([p])$ is well defined by Lemma
\ref{lem_evaluation}.
\subsection {Definition of the quantum Maslov
classes} \label{section:SeidelMorph}
Set
\begin{equation}\label{eq_gradedH}
  \pi _{} (\mathcal{P} (L _{0}, L _{1}), p _{0}) = \sum _{d
\text{ even}} \pi _{d}(\mathcal{P}
(L _{0}, L _{1}), p _{0}) \oplus \sum _{d \text{ odd}}
\pi_{d}(\mathcal{P}
(L _{0}, L _{1}  ), p _{0}),
\end{equation}
with the natural $\mathbb{Z} _{2}$ grading, where $p _{0}$
is any base point, which may be omitted.

We construct a graded group homomorphism:
\begin{equation} \label{eqPsiPaths}
\Psi: \pi (\mathcal{P} (L _{0}, L _{1} )) \to FH
(L _{0}, L _{1}),
\end{equation}
as follows.

Let $$f: S ^{n} \to  \mathcal{P} (L _{0},
L _{1}),
$$ be smooth with a chosen orientation on $S ^{n}$. (When
$n=0$, we set $S ^{0}$ to be a single point.) We may
suppose by homotopy approximation that each $f (s)$ is constant
in $[0, \epsilon] \cup [1-\epsilon, 1]$ for some
$0<\epsilon<1$.

We associate to this the data:
\begin{equation*}
   \{\mathcal{D} \times M, \mathcal{D}, \mathcal{L} _{s}\}
	 _{s \in S ^{n}} , 
\end{equation*}
$\mathcal{L}_{s}:=\mathcal{L} _{f (s)}  $ a Lagrangian subbundle of $M \times \mathcal{D
} $ over
$\partial \mathcal{D} $ determined by $f (s)$ as before. The end of $\mathcal{D}$ here is negative.

Now let $\mathcal{A} _{0} $ be a Hamiltonian connection on
$[0,1] \times M$. And let $\mathcal{A} _{0} (L _{0} ) \subset \{1\} \times M $ denote the $\mathcal{A} _{0} $-transport over $[0,1]$ of $L _{0} \subset \{0\} \times M $.
Suppose that $\mathcal{A} _{0} (L _{0} ) $ is transverse to
$L _{1} $. 

For each $s$ the space $\mathcal{C} _{s}$ of Hamiltonian connections
$\mathcal{L} _{s} $-exact with respect to $\mathcal{A} _{0}
$, (as in Section \ref{section:exact}) is affine and
non-empty and hence contractible. So we have
a Serre fibration over $\mathcal{C} \to S ^{n}$, with
contractible fibers.
Taking a section $\{\mathcal{A} _{s}\}$ of $\mathcal{C} $, we get an induced Hamiltonian structure:
$$ \{\Theta _{f,s}\}= \{\mathcal{D} \times M, \mathcal{D}, \mathcal{L} _{s}, \mathcal{A} _{s} \}$$
well-defined up to concordance, and we suppose that
$\mathcal{A} _{s}$ is chosen so that this is regular.

We then define $\Psi ([f])$ by:
\begin{equation*}
   \Psi ([f]) = [ev (\{\Theta _{f,s}\})],
\end{equation*}
with the right hand side as in Definition
\ref{def_EvaluationTotal}.
The expression for $\Psi$ is then well defined by Lemma \ref{lem_evaluation} since
the corresponding $\mathcal{K} = S ^{n} $ is a manifold with no
boundary.

\begin{lemma} \label{lem_PsiPigroup}
$\Psi$ determines a group homomorphism:
\begin{equation} \label{eqPsiPaths}
\Psi: \pi  (\mathcal{P} (L _{0}, L _{1})) \to FH (L _{0}, L _{1}).
\end{equation}
\end{lemma}
\begin{proof} [Proof] Let $f _{0}, f _{1}$  represent
classes $a, b \in \pi _{d} (\mathcal{P} (L _{0}, L _{1}  ),
p _{0})$. 
Since $\{\Theta _{f _{0},s}\}$ is regular, the set:
\begin{equation*}
\textbf{S}  = \cup _{A} \overline{\mathcal{M}} (\Theta _{f,s}, A)
\end{equation*}
is finite,
where the sum is over all $A$ that contribute
to $\Psi ([f])$ (recall that there are only finitely many of
them). 

Let $F: \textbf{S}  \to S ^{d}$ be the natural map sending an
element $(\sigma, s)$ to $s$.
By finiteness of $\textbf{S} $ we may find an 
$s _{0} \in  S ^{d}$ such that there is a open $U \ni s _{0}
\subset S ^{d}$, satisfying: $$F (\textbf{S} ) \cap U = \emptyset.$$

Let $s _{1} \in S ^{d}$ be the analogous element for the
family $\{\Theta _{f _{1},s}\}$. Let $$\{\Theta _{f _{0}
\# f _{1},b }\} _{b \in S ^{d} \# S ^{d}}$$ denote the gluing of the two Hamiltonian structures at $s _{0}, s _{1}$. Here we are just
doing a simple connect sum surgery, corresponding to the
connect sum surgery of the spheres $S ^{d} \# S ^{d} \simeq
S ^{d}$ at the points $s _{0}, s _{1}$.

By the defining properties of $s _{0}, s _{1}$, we may perform the surgery so that: 
\begin{equation*}
 ev (\{\Theta _{f _{0} \# f _{1}, b }\}) = ev (\{\Theta _{f
 _{0},s} \} + ev(\{\Theta _{f _{1}, s}\}.
\end{equation*}

On the other hand, the structure $\{\Theta _{f _{0} \#
f _{1}, b}\}$ is clearly concordant to the structure $\{\Theta
_{f _{0} + f _{1},s} \}$, where $f _{0} + f _{1} $ is the
concatenation sum in $\pi _{d} (\mathcal{P} (L _{0}, L _{1}  ),
p _{0})$. So 
\begin{equation*}
 ev (\{\Theta
_{f _{0} + f _{1},s} \}) = ev (\{\Theta _{f
 _{0},s} \} + ev(\{\Theta _{f _{1}, s}\},
\end{equation*}
and the result clearly follows.
\end{proof}

%
%

 
\subsection{Multiplicative structure} \label{sec_Multiplicative structure}
Let $C _{\bullet} Path (Lag (M))$ denote the $A _{\infty}
$ $\mathbb{Z} _{2} $-graded category over a commutative ring $k$, whose objects are elements of $Lag (M)$
and where $hom _{C _{\bullet} Path (Lag (M))}  ({L _{0},
L _{1}})$
is the $\mathbb{Z} _{2}$ graded $k$-module $$C(\mathcal{P} (L
_{0}, L _{1}  ), k) = \sum _{d
\text{ even}} C _{d}(\mathcal{P}
(L _{0}, L _{1}  ), k) \oplus \sum _{d \text{ odd}}
C_{d}(\mathcal{P}
(L _{0}, L _{1}  ), k).
$$  
The $A _{\infty} $ structure comes from the classical
Stasheff $A _{\infty}$ composition on the chain path
category. 
\begin{conjecture} \label{conj_functor}
The morphism $\Psi $ of Lemma \ref{lem_PsiPigroup} extends
to a $A _{\infty}$ functor
$$\widetilde{\Psi} : C
_{\bullet} Path (Lag (M)) \to Fuk (M),$$
where on objects $\widetilde{\Psi}  (L) = L$, and on
morphisms satisfying the following relation. For $a \in C(\mathcal{P} (L _{0}, L _{1}), k)$,
represented by $f: S ^{d} \to \mathcal{P} (L
_{0}, L _{1})$, $[\widetilde{\Psi} (a)] = \Psi ([f])$, where
the right-hand side is as in Lemma \ref{lem_PsiPigroup}.
\end{conjecture}

\section {A basic computation} \label{section:computeSeidel}
Under certain conditions  the spaces of perturbation data for certain problems in
Gromov-Witten theory admit a Hofer like functional. Although these spaces of
perturbations are usually contractible, there may be a gauge group in the
background that we have to respect, so that working equivariantly there is topology. The reader may think of the analogous situation in
Yang-Mills theory \cite{cite_AtiyahBottTheYang-MillsequationsoverRiemannsurfaces}.

The basic idea of the computation that we will perform consists of cooling the perturbation data as much as possible (in the sense
of the functional) to obtain a mini-max (for the functional)
data, using which we may write down our moduli spaces
explicitly. This idea was first used in
the author's \cite{cite_SavelyevVirtualMorsetheory}.
\subsection {Hofer length} \label{section:hoferlength} 
For $p: [0,1] \to \operatorname {Ham} (M,\omega)$  a smooth path, define
\begin{align*}
&  L ^{+} (p) :=   \int _{0} ^{1} \max _{M}  H ^{p} _{t} dt, \\
   &  L ^{-} (p) :=   \int _{0} ^{1} \max _{M}  (-H ^{p} _{t}) dt, \\
   & L ^{\pm} (p) :=   \max \{L ^{+} (p), L ^{-} (p)  \},  \\
\end{align*}
where $H ^{p}: M \times [0,1] \to \mathbb{R} $ generates
$p$, and is normalized by the condition that for each $t$, $H ^{p} _{t}:= H ^{p}| _{M \times \{t\}}   $ has mean 0, that is $\int _{M} H ^{p} _{t} dvol _{\omega} =0  $.

\subsubsection{Exact paths of Lagrangians}
\label{sec_ExactPaths}
Recall that an exact path $p$ of Lagrangians in $M$ is smooth
mapping $p: L  \times [0,1] \to M$, satisfying:
\begin{itemize}
	\item $p| _{L \times \{t\}}$
is an embedding for each $t$,  with Lagrangian image denoted $p (t)$.
\item The restriction of the 1-form $p ^{*} \omega
(\frac{\partial}{\partial t}, \cdot)$ to each $L \times \{t\}$  is exact.
\end{itemize}

If $p: [0,1] \to Lag (M)$ is an exact path,
define
\begin{equation*}
   L ^{+} _{lag}: \mathcal{P} Lag (M)  to \mathbb{R},     
\end{equation*}
\begin{equation*}
 L ^{+} _{lag}  (p) :=  \int _{0} ^{1} \max _{p (t)}  H ^{p } _{t} dt, 
\end{equation*}
$p (0) =L$ and where $H ^{p}: M \times [0,1] \to \mathbb{R}  $ is
normalized as above and generates a lift $\widetilde{p} $
of $p$ to $\operatorname {Ham} (M, \omega )$ starting at
$id$. By ``lift'', we mean that $p (t)=\widetilde{p} (t) (p (0)) $. (That is $H ^{p}  $ generates a path in
$\operatorname {Ham} (M,  \omega) $, which moves $L_0$ along
$p$.) Some theory of this latter functional is developed in
\cite{cite_LagrangianHofer}.
We may however omit the subscript $lag$ from notation, as usually there can be no confusion which functional we mean.

\subsubsection{Restriction to standard equators in $S ^{2}$} \label{sec_Lagrangian equators in $S ^{2}$}
Note that $Lag ^{eq}(S ^{2} ) $ is naturally diffeomorphic to $S ^{2} $ and
moreover it is easy to see that the functional $L ^{+}| _{Lag ^{eq}  (S ^{2} )} $ is proportional to the
Riemannian length functional $L _{met} $ on the path space of $S ^{2} $, with its
standard round metric $met$. We will then sometimes identify
$Lag ^{eq} (S ^{2})$ with $S ^{2}$.

Let now $L _{0}, L _{1} \in Lag ^{eq}  (S ^{2})$ be any
transverse pair, and  $$f': S ^{2} \to \mathcal{P} ^{eq} ({L
_{0},L _{1}}):= \{p \in \mathcal{P} ({L _{0},L _{1}})
\,|\, p (t) \subset Lag ^{eq}(S ^{2}), \, \forall t \in [0,1] \}, $$  
represent the generator $a$ of the group $\pi _{2}  
(\mathcal{P} ^{eq} ({L _{0},L _{1}}), \gamma) \simeq
\mathbb{Z}$, where $\gamma$  denotes the unique minimal
$met$-geodesic, from $L _{0}$ to $L _{1}$. (It is unique by
the assumption that $L _{0}, L _{1}$ are transverse, so that
the corresponding points on $S ^{2}$ are not conjugate with
respect to $met$.)

The idea of the computation is then this: perturb $f'$ to be transverse to the (infinite dimensional) stable manifolds
for the functional $L _{met}$ on $$\mathcal{P} ^{eq} ({L
_{0},L _{1}}), $$ then push the cycle down 
by the ``infinite time'' negative gradient flow for this
functional, and use the resulting representative to compute
$\Psi (a)$. Although, we will not actually need infinite
dimensional differential topology, instead we use a descent argument
using Whitehead's compression lemma.
\subsection {The ``energy'' minimizing perturbation data} 
\label{secion:pertubationdata}
Classical Morse theory \cite{cite_MilnorMorsetheory} 
tells us that the energy functional $$E (p) = \int _{[0,1]}
\langle \dot p (t), \dot p (t)  \rangle _{met}  \,dt  $$ on
$ \mathcal{P} ^{eq} ({L_0, L_1})$ is
Morse non-degenerate with a single critical point in each degree.
\begin{proposition} \label{thm_}
The homotopy class $a$ has a representative $f:
S ^{2} \to \mathcal{P} ^{eq} ({L_0, L_1}) $, such that:
\begin{enumerate}
	\item $f$ maps to  the
2-skeleton of $ \mathcal{P} ^{eq} ({L_0, L_1})$, for the
Morse cell decomposition induced by $E$.
\item   $f ^{*}E
$ is Morse, with a single  maximizer $\max$,  of index 2, and s.t. $\gamma _{0} = f (max)$ is the index $2$
geodesic.  
\end{enumerate}
We call such a representative $f$ \textbf{\emph{minimizing}}.
\end{proposition}
\begin{proof} [Proof]
This follows by Whitehead's compression lemma which is as follows.
\begin{lemma} [Whitehead, see \cite{cite_HatcherAlgebraic}]
   Let $(X,A)$ be a CW pair and let $(Y,B)$ be any pair with $B \neq \emptyset$. For each $n$ such that $X - A$ has cells of dimension $n$, assume that $\pi _{n} (Y,B,y _{0})=0$ for all $y _0 \in B$. Then every map $f:(X,A) \to (Y,B)$ is homotopic relative to $A$ to a map $X \to B$.
\end{lemma}
Suppose that $a$ has a representative 
$f': S ^{2} \to \mathcal{P} _{L_0, L_1} (S ^{2}) $ mapping into the $n$-skeleton $B ^{n} $ for the Morse cell decomposition for $E$, $n>2$.
Apply the lemma above with $(X,A) = (S ^{2}, pt )$, $Y = B ^{n} $ and $B = B ^{n-1} $ as above. Then the quotient $B ^{n} /B ^{n-1} $ is a wedge of $n$-spheres and since $\pi _{2} (S ^{n}) = 0 $ for $n>2$, $f$ can be homotoped into $B ^{n-1}   $ by the Whitehead lemma. Descend this way until we get a representative mapping into $B ^{2} $.

Furthermore since
$\pi_2 (S ^{1} ) = 0 $  such a representative cannot
entirely lie in the 1-skeleton. 
It follows, since we have a single Morse 2-cell, that 
there is a representative $f: S ^{2} \to  \mathcal{P} ^{eq}
({L_0, L_1})  $, for $a$, s.t. the function ${f} ^{*} E  $  is Morse with a 
maximizer $\max$,  of index 2, and s.t. $\gamma _{0} = f (max)$ is the index $2$
geodesic.  

In principle there maybe more than one such maximizer $\max$, but recall
that we assumed that $a$ is the generator, so the degree of the
map $f: S ^{2} \to B ^{2}/B ^{1} \simeq S ^{2}$ is one. 
By the Hopf theorem, homotopy classes of maps  of spheres
are classified by degree. Hence after a further homotopy, 
there will be only one such maximizer.
\end{proof}

It follows that $\max$ is likewise the unique index 2 maximizer of the function ${f} ^{*} L _{met}  $ by the classical relation between the energy functional and length functional. And so $\max$ is the index 2 maximizer of $f ^{*}L ^{+}  $.
\subsection {The corresponding minimizing data} \label{section:distinguisheddata} 

\begin{lemma} \label{lemma:taut} There is a minimizing
representative $f _{0} $ for the class $a$ and a taut
Hamiltonian structure $$\Theta _{f _{0}}  =\{\mathcal{D}
\times S ^{2}, \mathcal{D}, \mathcal{L} _{f _{0} (b)},
\mathcal{A} _{b}  \},$$ 
satisfying:
\begin{enumerate}
	\item $f _{0} (b) \in Lag ^{eq} (S ^{2})$, for each $b$.
	\item \begin{equation} \label{eq.arealength}
  \forall b: \area (\mathcal{A} _{b}) = L ^{+} (f _{0}  (b)).
\end{equation} 
\end{enumerate}

\end{lemma}
\begin{proof}
Note that a geodesic segment $p: [0,1] \to S ^{2} \simeq Lag
^{eq} (S ^{2})$ for the
round metric $met$ on $S ^{2} $ has a unique lift
$$\widetilde{p}: [0,1] \to PU (2) \simeq SO (3),$$
$\widetilde{p} (0)=id $ with $\widetilde{p} $ a segment of
a one parameter subgroup, and in this case $$L ^{+} _{lag}
(p)= L ^{+} (\widetilde{p}). $$ It then follows that for
a piecewise geodesic path $p$ in $S ^{2} $, there is
likewise a unique lift $ \widetilde{p}: [0,1] \to PU (2), $ satisfying $$L ^{+} _{lag} (p)= L ^{+} (\widetilde{p}). $$

Now, if $f$ is a minimizing representative of $a$ as above, we may homotop it to a likewise minimizing representative $f _{0}
$, so that for all $b$ $f _{0}  (b)$ is piecewise geodesic.
This follows by the piecewise geodesic approximation theorem
Milnor~\cite[Theorem 16.2]{cite_MilnorMorsetheory} of the loop space. 

Let $\mathcal{A} _{0} $ be the trivial Hamiltonian
connection on $[0,1] \times M$. Use the construction of
Lemma \ref{lemma:areadisk}, to get a family of Hamiltonian
connections $\{ \widetilde{\mathcal{A}} ^{f _{0}  (b)} _{0} \}   $. In
this case, since $\mathcal{A} _{0} $ is trivial $$\area
(\widetilde{\mathcal{A}} ^{f _{0}  (b)} _{0}  ) = L ^{+}(f _{0}  (b)). $$  

Set $A _{b} = \widetilde{\mathcal{A}} ^{f _{0}  (b)} _{0}$.
It remains to verify that $\Theta _{f _{0} }  =\{\mathcal{D}
\times S ^{2}, \mathcal{D}, \mathcal{L} _{f _{0} (b)},
\mathcal{A} _{b} \}$ is taut, which also implies
\eqref{eq.arealength}.
This follows by the following lemma.  
\begin{lemma} \label{lemma:tautLagS2}  Two loops $p _{0},
p _{1}: S ^{1} \to  Lag ^{eq}  (S ^{2}) \subset Lag (S ^{2})
$ are taut concordant as defined in the Definition \ref{def_concordantIntro}, iff
they are homotopic in $Lag ^{eq} (S ^{2})$. 
\end{lemma}
\begin{proof}
Suppose that $H: S ^{1} \times [0,1] \to Lag ^{eq} (S ^{2})$ is a smooth homotopy between $p _{0},
p _{1}$. Let $\mathcal{L}$ be the corresponding Lagrangian
sub-fibration of $(Cyl = S ^{1} \times [0,1]) \times S ^{2}
$. That is for $(\theta, t) \in Cyl$, $\mathcal{L}
_{(\theta, t)} = H (\theta, t)$.

Then $\mathcal{L} $ is as in Definition \ref{def_concordantIntro}. Let
$\mathcal{A}$ be any $PU (2)$ connection on $Cyl \times \mathbb{CP} ^{1}$ which preserves $\mathcal{L}$ (there are no obstructions to constructing this). Then $R _{{A}}$ is a $\lie PU (2)$ valued 2-form, such that for all $v,w \in T _{z} Cyl $ the vector field $X=R _{\mathcal{A}} (z) (v,w)$ is tangent to $\mathcal{L} _{z} $. 

Let $\phi: (M _{z}, L _{z}) \to (\mathbb{CP} ^{1}, L _{0})$
be a Kahler map of the pair, where $L _{0}$  is the
standard equator. Then $\phi _{*} (X)$ is tangent to $L
_{0}$ and is in the image of $$\mathfrak{lie} (PU (2)) \to
\mathfrak{lie} (\operatorname {Ham} (\mathbb{CP} ^{1},
\omega _{st}) ) \simeq \mathbb{R} ^{}.  $$ In particular, 
its generating Hamiltonian is some standard height
function on $\mathbb{CP} ^{1}$ and so vanishes on $L _{0}$.
Thus $H _{X}$ vanishes on $L _{z}$.
By the definition, $\Omega _{\mathcal{A}} $ vanishes on $\mathcal{L}$ and so we are done.
\end{proof}
\end{proof}

So given $\{\mathcal{A} _{b} \}$ as in the lemma above, since 
$$\forall b: \area (\mathcal{A}  _{b}  ) = L ^{+} (f _{0}  (b)),$$ we immediately deduce:
\begin{lemma} \label{lemma.morseatmax} The function $area: b \mapsto \area
   (\mathcal{A}  _{b}  )$ has
   a unique maximizer, coinciding with the maximizer $\max$ of ${f} _{0}  ^{*}L
   _{met}  $ and $\area$ is Morse at $\max$ with index $2$.
\end {lemma}
\subsection {Holomorphic sections for the data}
Let us now rename $f _{0} $ by $f$, $\mathcal{L} _{f _{0} (b) } $ by $\mathcal{L} _{b} $, and $\Theta _{f _{0} }  $ by $\Theta = \{\Theta _{b} \}$.

   As $ f (\max)$ is a geodesic for $met$,  its lift
	 $\widetilde{f} (\max) $ to $SO (3)$ is a rotation around
	 an axis intersecting $L _{0} = f (\max) (0)$ in a pair of
	 points, in particular there is a unique point $$x _{\max}
	 \in \bigcap _{t} (L _{t} = f (\max) (t))   $$ maximizing
	 $H ^{\max} _{t} $ for each $t$. In our case this follows
	 by elementary geometry but there is a more general
	 phenomenon of this form c.f. \cite{cite_LagrangianHofer}. 

Define  $$\sigma  _{\max}: \mathcal{D}  \to  \mathcal{D} \times S ^{2}   $$ to be the constant section $z \mapsto x _{\max }.
$ 
Then $\sigma _{\max} $ is a $\mathcal{A}  _{\max}  $-flat
section with boundary on $\mathcal{L}_{\max} $, and is consequently $J (\mathcal{A} _{\max} )$-holomorphic.

\begin{lemma}
$\sigma _{\max}$ satisfies: $$Maslov ^{vert} (\sigma _{\max}
^{/}) = -2,$$ see Appendix
\ref{appendix:maslov} for the definition of this Maslov
number.
\end{lemma}
\begin{proof}
 Set $$T ^{vert} _{z} \mathcal{L} _{\max}:= \{v \in T \mathcal{L} \subset T _{z} (\mathcal{D} \times S ^{2})\,|\, pr _{*} v =0 \}   $$ where $pr: \mathcal{D} \times S ^{2} \to \mathcal{D} $ is the projection. 
Denote by $$Lag (T _{x _{\max}} S ^{2} \simeq Lag (\mathbb{R} ^{2} ) \simeq S ^{1}  $$ the space of oriented linear Lagrangian subspaces of $T _{x _{\max} } S ^{2} $. 
Let $\rho$ be the path in $Lag (T _{x _{\max}} S ^{2})$ defined by 
\begin{equation*}
\rho (t) =  T ^{vert} _{(\zeta (t), x _{\max} )}   \mathcal{L} _{\max}, \quad t \in [0,1]
\end{equation*}
where $\zeta: \mathbb{R} \to \partial \mathcal{D}$ is a fixed parametrization as in Definition \ref{defInducedLoop}. 

By our conventions  for the Hamiltonian
vector field: $$\omega (X _{H}, \cdot ) = - dH (\cdot),$$
$\rho$ is a clockwise oriented path from $$T _{x _{\max} } L _{0} := T ^{vert} _{(\zeta (0), x _{\max} )} \mathcal{L} _{\max} $$ to $$T _{x
_{\max} } L _{1} := T ^{vert} _{(\zeta (1), x _{\max} )} \mathcal{L} _{\max}  $$ for the orientation induced by the complex orientation on $T _{x _{\max} } S ^{2}  $. 

By the Morse index theorem in Riemannian geometry
\cite{cite_MilnorMorsetheory} and by the condition that $f (\max)$ has Morse index 2, $\rho$  visits initial point $\rho (0)$ exactly twice if we count the start, as this corresponds to the geodesic $f (\max)$ passing through two conjugate points in $S ^{2} $. So the concatenation of $\rho$ with the minimal counter-clockwise path from $T _{x
_{\max} } L _{1}   $ back to $T _{x
_{\max} } L _{0}  $ is a degree $-1$ loop, if $S ^{1}  \simeq  Lag
(\mathbb{R}^{2})$ is given the counter-clockwise
orientation. Consequently,
   $$Maslov ^{vert} (\sigma _{\max} ^{/}) = -2.  $$  
\end {proof}
We set $A _{0} = [\sigma _{\max}]$.
\begin{proposition} \label{lemma:onlyelement}
   $(\sigma _{\max}, \max)$ is the sole element of
$
   \overline {\mathcal{M}} (\Theta, A _{0}).
$
\end{proposition}
\begin{proof}
By Stokes theorem, since $\omega$ vanishes on $\sigma _{\max} $, it is immediate:
\begin{equation} \label{eqLength1}
\carea (\Theta _{\max}, A _{0})=- \int _{\mathcal{D}} \sigma _{\max}  ^{*}  \widetilde{\Omega} _{\max}
= L ^{+} (f (\max)). 
\end{equation}
   Moreover, since $\Theta=\{\Theta _{b} \} $ is taut
	 $\carea (\Theta _{b}, A _{0}  ) =  L ^{+} (f (\max))$.
So by \eqref{eq.arealength} and by Lemmas \ref{lemma:energypositivity1},
   \ref{lemmaInvariantPairing} we have:
\begin{equation*}
 L ^{+} (f (\max))  \leq \area
(\mathcal{A}_{b}) = L ^{+}(f (b)),
\end{equation*}
whenever there is an element 
$$(\sigma, b) \in  
\overline {\mathcal{M}} (\{
   \Theta _{b}  
\}, A _{0}).
$$
But clearly this is impossible unless $b= \max$, 
since $L ^{+} (f (b) ) < L ^{+} (f (\max) )  $
for $b \neq \max$. So to finish the proof of the proposition we just need:
\begin{lemma}
There are no elements $\sigma$ other than $\sigma
_{\max}$ of
the moduli space 
$$
\overline {\mathcal{M}} (
   \Theta _{\max}, A _{0}).
$$ 
\end{lemma}
\begin{proof}
We have by \eqref{eqLength1}, and by \eqref{eq.arealength}
$$0=  \langle [\widetilde{\Omega} _{\max} +
   \alpha _{\widetilde{\Omega}_{\max}}] , [\sigma_{\max} ]
   \rangle, $$
and so given another element $\sigma$   
we have:
\begin{equation*} \label{eq:vanishing}
   0=    \langle [\widetilde{\Omega}   _{\max} +
   \alpha _{\widetilde{\Omega} _{\max}}] , [\sigma] \rangle.
\end{equation*}
It follows that $\sigma $ is necessarily $\widetilde{\Omega}   _{\max}
$-horizontal, since 
$$(\widetilde{\Omega}   _{\max} +
\alpha _{\widetilde{\Omega}   _{\max}  }) (v, J _{\widetilde{\Omega}  _{\max} }v) \geq 0.$$
Since $J _{\widetilde{\Omega}  _{\max}}$ by assumptions preserves the vertical and $\mathcal{A} _{\max} $-horizontal subspaces of $T (\mathcal{D} \times S ^{2} )$, and since the 
inequality is strict for $v$ in the vertical tangent bundle of
$$S ^{2} \hookrightarrow \mathcal{D} \times S ^{2}   \to \mathcal{D},
$$ 
the above inequality is strict whenever $v$ is not horizontal.
So $\sigma$ must be $\mathcal{A} _{\max} $-horizontal. But then $\sigma = \sigma _{\max} $ since $\sigma _{\max} $ is the only flat
section asymptotic to $\gamma _{0} $.
\end {proof}
\end {proof} 
\subsubsection {Regularity}
Since by the dimension formula \eqref{eq:dimensionMaslov}
only class $A _{0}$ curves can contribute to $\Psi (a)$, it will follow that $$\Psi (a) = \pm [\gamma _{0} ]$$ if we knew
that
$(\sigma_{\max}, \max)$ was a regular element of 
$$ 
\overline {\mathcal{M}} (\{
   \Theta _{b} 
\}, A _{0}).
$$
We won't answer directly if $(\sigma _{\max}, \max)$ is regular, 
although it likely is. But
it is regular after a suitably  small Hamiltonian perturbation of the family
$\{\mathcal{A}  _{r} \}$ vanishing at $\mathcal{A}  _{\max} $.  
We call this essentially automatic regularity. 
\begin{lemma} \label{lemma.perturbedreg} 
There is a family $\{\mathcal{A} ^{reg}_{b}  \}  $ arbitrarily $C ^{\infty}
$-close to $\{\mathcal{A} _{b} \}$ with  $\mathcal{A} ^{reg}_{\max} =
\mathcal{A} _{\max}  $ and such that
\begin{equation} \label{eqModuliReg} 
\overline {\mathcal{M}} (\{
   \mathcal{D} \times S ^{2}, \mathcal{D}, \mathcal{L} _{b}, \mathcal{A} ^{reg}  _{b} 
\}, A _{0}),
\end{equation} 
is regular, with $({\sigma}_{\max}, {\max}  )$  its sole element. In particular $$\Psi (a) = \pm [\gamma _{0} ].$$
\end{lemma}
\begin{proof}
The associated real linear Cauchy-Riemann operator $$D_{\sigma_{\max}}: \Omega ^{0} (\sigma _{\max} ^{*} T
   ^{vert} \mathcal{D} \times S ^{2} _{\max})    \to \Omega ^{0,1} (\sigma _{\max} ^{*} T
^{vert} \mathcal{D} \times S ^{2} _{\max}),$$ 
has no kernel, by Riemann-Roch 
\cite[Appendix
C]{cite_McDuffSalamonJholomorphiccurvesandsymplectictopology}, 
as the vertical Maslov number of $[\sigma _{\max} ]$ is $-2$.
And the Fredholm index of
$({\sigma}_{\max}, {\max}  )$ which is -2,  is -1 times the Morse index of the function $\area$
at $\max$, by
Lemma \ref{lemma.morseatmax}.
Given this, our lemma follows by a direct translation of
\cite[Theorem
1.20]{cite_SavelyevMorsetheoryfortheHoferlengthfunctional}, itself
elaborating on the argument in
\cite{cite_SavelyevVirtualMorsetheory}.
\end{proof}
To summarize: 
\begin{theorem} \label{thmNonZeroPsi} For $0 \neq b \in
\pi _{2} (\mathcal{P} ^{eq} ({L _{0}, L _{1}}))$,
and	$a = i _{*} (b)$, where $i: \mathcal{P} ^{eq} ({L _{0},
L _{1}}) \to \mathcal{P}  ({L _{0}, L _{1}})$ is the
inclusion.
\begin{equation*}
  0  \neq  \Psi (a) \in HF (L _{0}, L _{1}).
\end{equation*}
\end {theorem} 
\begin{proof} We have shown that $0  \neq  \Psi (a) \in HF
(L _{0}, L _{1})$, for $a$ the image of the generator of the
group $\pi _{2} (\mathcal{P} ^{eq} (L _{0}, L _{1}),
\mathbb{Z}).$ Since $\Psi$ is an additive group homomorphism the conclusion follows.
\end{proof}
\begin{proof} [Proof of the second part of Theorem \ref{thm:Hofer}
]
By the theorem above the natural map $$\pi _{2} (\mathcal{P}
^{eq} (L _{0}, L _{1}) \to  \pi _{2} (\mathcal{P} (L _{0},
L _{1})),$$ is an injection and hence:
\begin{equation*}
  \pi _{2} (\mathcal{P}
^{eq} (L _{0}, L _{0}) \to  \pi _{2} (\mathcal{P} (L _{0},
L _{0})),
\end{equation*}
is an injection and hence:
\begin{equation*}
  \pi _{3} (Lag ^{eq} (S ^{2})) \to  \pi _{3} (Lag (S
	^{2})),
\end{equation*}
is an injection.
\end{proof}

\section{Proof of Theorem \ref{thm:Hofer}} \label{section:applicationHofer}
Suppose otherwise, so that $$\min_{f \in  b} \max _{s \in S ^{2} } L ^{+} (f (s)) = U < \hbar,$$ for $b$ as in the statement of the theorem. 
	 
Fix $L _{1} \in Lag ^{eq} (S ^{2} ) $ so that $L _{0}
$ intersects $L _{1} $ transversally, and so that there is
a minimal geodesic path $p _{0} \in \mathcal{P} Lag ^{eq}  (L _{0}, L _{1}  ) $  with $$\kappa := L ^{\pm} (\widetilde{p} _{0} )   < \epsilon = (\hbar - U)/2. $$ Here $\widetilde{p} _{0}  $ is the geodesic lift to $PU (2)$ starting at $id$. 
	 
Then concatenating $f (s)$ with $p _{0} $ for each $s$, we obtain a smooth
family of paths: 
\begin{align*}
   & g: S ^{2} \to \mathcal{P} (L _{0},L _{1}) \\
   & g (0) = p _{0},
\end{align*}
such that $g$ represents the previously appearing class $a$, that
is the generator of the group $$\pi _{2}(
\mathcal{P} (L _{0},L _{1}), p _{0}). $$ 

Let $$\{\Theta _{s} \}= \{
   \mathcal{D} \times S ^{2}, \mathcal{D}, \mathcal{L} _{s}, \mathcal{A}_{s} 
\} _{\mathcal{K} = S ^{2}} ,$$ be the corresponding
Hamiltonian structure, where $\mathcal{A} _{s}$ is as in
Lemma \ref{lemma:taut}, defined with respect to $g$, and where $\mathcal{L} _{s}:=\mathcal{L} _{g (s)} $. 

In particular, $\{\Theta _{s} \}$ is taut and satisfies:
\begin{equation} \label{eq:hbarlower}
\forall s \in S ^{2}: \area (\mathcal{A} _{s} ) = L ^{+} (g (s))  < \hbar - \kappa.
\end{equation}

By the assumption that each $g (s)$ is taut concordant to the
constant loop at $L _{0} $, each $\Theta _{s} $ is taut
concordant to $$\mathcal{H} = (\mathcal{D} \times S ^{2},
\mathcal{D}, \mathcal{L} _{0}, \mathcal{A}),$$ where
$\mathcal{L} _{0} = \mathcal{L} _{p _{0}}  $,
${\mathcal{A}} = \widetilde{\mathcal{A}} _{0} ^{p _{0}}
$, where $\widetilde{\mathcal{A}} _{0} ^{p _{0}} $   is as
in Lemma \ref{lemma:areadisk}, for $\mathcal{A} _{0} $ the
trivial connection. 

Let $\Theta (L, \mathcal{A} _{0})$ be the construction as in \eqref{eq:ThetaL}.
Then  by Lemma \ref{lemma:immediategluing}, for each $s$, $$\Theta _{s} ^{/0}
: = \Theta _{s} \# _{0} \Theta (L _{0},
\mathcal{A} _{0} ) $$  is taut concordant to
$$\mathcal{H} ^{/0} : = \mathcal{H}  \# _{0} \Theta (L _{0},
\mathcal{A} _{0} ). $$ 

On the other hand, by Lemma \ref{lemma:tautLagS2}
$\mathcal{H} ^{/0} $ is taut concordant to the trivial
Hamiltonian structure $(D ^{2}  \times S ^{2}, D ^{2},
\mathcal{L} _{tr} , \mathcal{A} _{tr} )$, where $\mathcal{L} _{tr} $ the trivial bundle with fiber $L _{0} $ and $\mathcal{A} _{tr} $ the trivial Hamiltonian connection. 
So for each $s$:
\begin{equation} \label{eq:careablah}
\carea (\Theta _{s} ^{/0} , A _{0}) =  \carea (\mathcal{H}
^{/0} , A _{0}) = \hbar.
\end{equation}

Now by Theorem \ref{thmNonZeroPsi} $$ev (\{\Theta _{s} \}, A _{0} ) = \Psi (a) \neq 0.$$ And so:
\begin{equation*}
\overline {\mathcal{M}} (\{\Theta _{s}\}, A _{0} ) \neq \emptyset,
\end{equation*}
but this contradicts the conjunction of \eqref{eq:hbarlower}, \eqref{eq:careablah},  and Lemma \ref{lemma:gluinglowerbound}.
\qed

\appendix 
\section{On the Maslov number and dimension formula} \label{appendix:maslov}
Let ${S}$ be a Riemann surface with boundary and a strip end
structure as previously.

Let $\mathcal{V} \to S$ be a rank $r$ complex vector bundle,
trivialized at the open ends $\{e _{i} \}$, so that we have distinguished bundle charts $
[0,1] \times (0,\infty) \times \mathbb{C} ^{r}  \to \mathcal{V}$ at the positive ends. (Similarly, for negative ends.) 

Let $$\Xi \to \partial S  \subset S $$
be a totally real rank $r$ subbundle of $\mathcal{V}$, which  is constant in the
coordinates $$ [0,1] \times (0, \infty) \times \mathbb{C}^{r},  $$
at the positive ends, again similarly with negative ends. 

For each (positive) end $e _{i} $ and its chart  $e _{i}:
[0,1] \times (0, \infty) \to S$, let
$b ^{j}  _{i}: (0, \infty) \to \partial S  $, $j=0,1$ be the restrictions of $e _{i} $ to $\{i\} \times (0, \infty)$.

We then have a pair of real vector spaces $$\Xi _{i} ^{j} =  \lim _{\tau \mapsto \infty} \Xi|_{b
^{j} _{i} 
 (\tau) }.
 $$

There is a Maslov number $Maslov (\mathcal{V}, \Xi, \{\Xi
^{j}  _{i} \})$ associated to this data, and which we now
briefly describe.
In the case 
$\Xi _{i} ^{0} = \Xi _{i} ^{1}$, let $(\mathcal{V} ^{/}, \Xi
^{/}  )$ be obtained from $(\mathcal{V}, \Xi,
\{\Xi ^{j}  _{i} \})$ by capping off each $e _{i} $ end of $\mathcal{V} \to S$. Here the capping operation is similar to the one in Section
\ref{section:gluing}.
Then $Maslov (\mathcal{V}, \Xi, \{\Xi ^{j}  _{i} \})$
coincides with the boundary Maslov index of $(\mathcal{V} ^{/}, \Xi ^{/}  )$ in the sense of
\cite[Appendix
C3]{cite_McDuffSalamonJholomorphiccurvesandsymplectictopology}.

When $\Xi _{i} ^{0}$ is transverse to $\Xi _{i} ^{1}$ for each $i$,  $Maslov (\mathcal{V}, \Xi, \{\Xi ^{j}
_{i} \})$ is obtained as the Maslov index for the
modified pair $(\mathcal{V} ^{/}, \Xi ^{/}  )$
obtained by again capping off the ends $e _{i} $ via gluing (at each end $e _{i} $) with $$(\mathcal{D} \times \mathbb{C}^{r}, 
 \widetilde{\Xi}, \{\widetilde{\Xi}_{0} ^{j}  \} ),$$ where
 $\mathcal{D}$ is as before.  Here  
 $\widetilde{ {\Xi}} ^{0} _{i} = { {\Xi}} ^{1} _{0} $ and 
$\widetilde{ {\Xi}} ^{1} _{i} = { {\Xi}} ^{0} _{0} $, while
$ \widetilde{\Xi}$ over the boundary of $\mathcal{D}$ is determined by the ``shortest path'' from 
$\widetilde{ {\Xi}} ^{0} _{0}$ to $\widetilde{ {\Xi}} ^{1}
_{0}$, which means the following. 
As $\widetilde{ {\Xi}} ^{0} _{0}$ to $\widetilde{ {\Xi}} ^{1}
_{0}$ are a pair of transverse, totally real subspaces, up
to a complex isomorphism of $\mathbb{C} ^{r} $ (whose choice
will not matter), we may identify them with the subspaces
$\mathbb{R} ^{r} $, and $i \mathbb{R} ^{r} $. After this identification our shortest path is just $e ^{i \theta} \mathbb{R} ^{r}  $, $\theta \in [0, \pi _{2} ]$.

Let $D$ be a real linear Cauchy-Riemann operator on
$\mathcal{V}$, which in particular is an operator: 
\begin{equation*}
D: \Omega ^{0} _{\Xi} (S, \mathcal{V}) \to \Omega ^{0,1}
_{\Xi} (S, \mathcal{V}),  
\end{equation*}
where $\Omega ^{0} _{\Xi} (S, \mathcal{V})$ denotes the
space of smooth $\mathcal{V} $-valued $0$-forms (i.e.
sections) satisfying $\theta (\partial S) \subset \Xi $, and
$\Omega ^{0,1} _{\Xi} (S, \mathcal{V})$ denotes the
analogous space of smooth complex anti-linear 1-forms.

Suppose further that $D$ is asymptotically $\mathbb{R} ^{}
$-invariant in the strip end coordinates at the ends.
After standard Sobolev completions,  the Fredholm index of
$D$ is given by:
\begin{equation*}
r \cdot \chi (S) + Maslov (\mathcal{V}, \Xi, \{\Xi _{i} \}).
\end{equation*}
The proof of this is analogous to \cite[Appendix
C]{cite_McDuffSalamonJholomorphiccurvesandsymplectictopology}, we can also reduce it to that statement via a gluing
argument. (This kind of argument appears for instance in
\cite{cite_SeidelFukayacategoriesandPicard-Lefschetztheory})
\subsection {Dimension formula for moduli space of sections}
Suppose we have a Hamiltonian structure $\Theta = (\widetilde{S}, S,
\mathcal{L}, \mathcal{A} ) $.  Suppose that either the
corresponding Lagrangian submanifolds $$L _{i} ^{j} =  \lim _{\tau \mapsto \infty} \mathcal{L}|_{b
^{j} _{i} (\tau) },$$ intersect transversally (identifying
the corresponding fibers) or coincide. (Similarly for
negative ends.) 

Let $A \in H _{2} ^{sec} (\widetilde{S}, \mathcal{L} ) $,
with the latter as in Section \ref{def:HamiltonianStructure}. And let $\mathcal{M} (\Theta, A)$ be as in Section \ref{def:HamiltonianStructure}. Define $$Maslov ^{vert} (A) $$ to be the 
Maslov number of the triple $(\mathcal{V}, \Xi, \{\Xi _{i} \})$ determined by the pullback by $\sigma \in \mathcal{M} (A)$ of the vertical tangent bundle of $\widetilde{S}$, $\mathcal{L}$.
Then the expected dimension of $\mathcal{M} (A)$ is:
\begin{equation} \label{eq:dimensionMaslov}
r \cdot \chi (S) + Maslov ^{vert}  (A).
\end{equation}
\bibliographystyle{siam}
\bibliography{link.bib}
\end{document}